\documentclass[english]{smfart}
\usepackage{fullpage}
\usepackage[utf8x]{inputenc} 
\usepackage[T1]{fontenc}
\usepackage{amsmath,amsfonts,latexsym,tikz, amsthm}
\usepackage{url}
\usepackage[greek, french, english]{babel}
\usepackage{epigraph}
\usepackage{multicol}

\usepackage{color}

\usepackage{tikz}
\usetikzlibrary{decorations.pathmorphing, decorations.markings, arrows, matrix, fit, positioning, chains}

\usepackage{tikz-cd}

\renewcommand{\l}{\left[}
\renewcommand{\r}{\right] }

\theoremstyle{plain}
\usepackage{cancel}

\newtheorem{teo}{Theorem}[section]
\newtheorem{coro}[teo]{Corollary}
\newtheorem{defin}[teo]{Definition}

\newtheorem{lema}[teo]{Lemma}
\newtheorem{prop}[teo]{Proposition}
\newtheorem{rem}[teo]{Remark}

\newtheorem{nota}[teo]{Notation}
\newtheorem{slema}[teo]{Sub-lemma}
\newtheorem{hyp}[teo]{Hypothesis}

\newcommand{\N}{\mathbb{N}}
\newcommand{\Z}{\mathbb{Z}}
\newcommand{\R}{\mathbb{R}}

\DeclareMathOperator{\Id}{Id}

\newcommand{\Ch}{\mathbf{C(k)}}
\newcommand{\dg}{\mathbf{dg-cat}}
\newcommand{\sS}{\mathbf{sSet}}
\newcommand{\Set}{\mathbf{Set}}

\newcommand{\sC}{\mathbf{sCat}}

\newcommand{\Fact}{\mathbf{Fact}}
\newcommand{\Free}{\mathcal{L}}
\newcommand{\FreeS}{c\mathcal{L}_\S}

\newcommand{\sAb}{\mathbf{sAb(k)}}

\newcommand{\D}{\Delta_k}
\newcommand{\0}{\emptyset}

\newcommand{\dgS}{\mathbf{dg-Segal}}
\newcommand{\dgSc}{\mathbf{dg-Segal_c}}
\renewcommand{\k}{\underline{k}}

\renewcommand{\S}{\mathbb{S}}

\DeclareMathOperator{\Ob}{Obj}

\DeclareMathOperator{\Fun}{Fun}
\DeclareMathOperator{\Map}{Map}
\DeclareMathOperator{\Ho}{Ho}
\newcommand{\op}{^{op}}
\DeclareMathOperator{\Hom}{Hom}

\renewcommand{\O}{\mathcal{O}}
\DeclareMathOperator{\map}{map}

\DeclareMathOperator{\colim}{colim}
\renewcommand{\L}{\mathbb{L}}

\DeclareMathOperator{\hocolim}{hocolim}
\DeclareMathOperator{\holim}{holim}
\DeclareMathOperator{\Alg}{Alg}

\renewcommand{\Re}{\operatorname{Re}}
\DeclareMathOperator{\Sing}{Sing}
\DeclareMathOperator{\sk}{sk}
\DeclareMathOperator{\cosk}{cosk}

\author{Elena Dimitriadis Bermejo}
\email{dimitriadis-bermejo.math@posteo.net}
\urladdr{\url{https://www.math.univ-toulouse.fr/~edimitri/index_en}}

\title{A new model of dg-categories}
\alttitle{Un nouveau modèle des dg-catégories}

\begin{document}

\frontmatter

\begin{abstract}
In this article, we develop a new model for the category of dg-categories. Following Rezk's example in the case of classic Segal spaces, we define dg-Segal spaces: functors between free dg-categories of finite type and simplicial spaces to which we add certain properties. We define also complete dg-Segal spaces, and make their relationship to classic Segal spaces explicit. With the help of two new hypercover constructions, and up to a certain hypothesis, we prove that there exists an equivalence between the homotopy category of dg-categories and the homotopy category of functors defined above with a model structure making the complete dg-Segal spaces into its fibrant objects. 
\end{abstract}

\begin{altabstract}Dans cet article, on développe un nouveau modèle de la catégorie des dg-catégories. En suivant l'exemple de Rezk dans le cas des espaces de Segal classiques, on définit des espaces de dg-Segal: c'est-à-dire, des foncteurs entre les dg-catégories libres de type fini et les espaces simpliciaux, auxquels on ajoute certaines conditions. On définit aussi des espaces de dg-Segal complets, et on explicite leur relation avec les espaces de Segal classiques. Avec l'aide de deux nouvelles constructions d'hyperrecouvrement, et en acceptant une certaine hypothèse, on prouve qu'il existe une équivalence entre la catégorie homotopique des dg-catégories et la catégorie homotopique des foncteurs définis auparavant, avec une structure de modèles qui fait des espaces de dg-Segal complets ses objets fibrants.
\end{altabstract}


\altkeywords{dg-catégories, hyperrecouvrements, modèles alternatifs pour des dg-catégories, espaces de Segal complets, catégories de modèles, catégories de foncteurs simpliciaux, dg-catégories libres, dg-catégories de type fini}

\date{}

\maketitle
\tableofcontents

\mainmatter

\section{Introduction}

\subsection{A historical perspective}

Dg-categories (i.e. categories enriched over cochain complexes) have been extremely useful over the years in Algebraic Geometry. Indeed, although classical categories are enough for Algebraic Topology, Algebraic Geometry deals with a notion of linearity that does not gel well with those. As such, authors in this field have turned to dg-categories as their generalization of derived categories. But as useful as these are, they are not without issues. In particular, one glaring problem with dg-categories is the fact that dg-categories have a model structure and a monoidal structure, but the two of them are not compatible.

People working on dg-categories are not the first ones to encounter that problem, though. Indeed, authors on $\infty$-categories found themselves dealing with a similar issue in the 80s: simplicial categories, one of the first models for those, also have a model structure that isn't compatible with its monoidal structure. In order to solve that, other models were suggested over the years, each with their own advantages and disadvantages, and were proven to be all Quillen equivalent. Four models, in particular, are the most useful and the most used: simplicial categories, quasi-categories, Segal categories and complete Segal spaces. In recent years, authors have turned repeatedly to those for inspiration for new models for dg-categories.

 In 2013, Cohn proved in \cite{COHN} that there is an equivalence between the underlying $\infty$-category of the model category of dg-categories localized at the Morita equivalences, and the $\infty$-category of small idempotent-complete $k$-linear stable quasi-categories. In 2015, Gepner and Haugseng defined in \cite{Gepner-enriched-quasi} quasi-categories enriched over a "nice" Cartesian category $V$, and a particular case of Haugseng's results states that dg-categories are equivalent (as a quasi-category) to quasi-categories enriched in the derived quasi-category of abelian groups. On another attempt in the same direction, Mertens constructed in 2022  a version of enriched quasi-categories, and used it to construct a linear version of the classical dg-nerve functor (see \cite{Arne-dg-quasi-cat} and \cite{Arne-dg-nerve}). In this case, the model structure and the proof of them being Quillen equivalent to dg-categories are still in progress. On the Segal category side of things, in 2020 Bacard defined in \cite{Bacard-enriched-Segal-cat} a concept of enriched Segal categories, but as with Mertens, the model structure and the equivalence seem to be a future project.

\begin{center}
\begin{tabular}{|c|c|}
\hline
simplicial categories  & dg-categories \\
\hline
quasi-categories \cite{Boardman-Vogt-quasicat} & $k$-linear quasi-categories \cite{COHN} \\
						& enriched quasi-categories \cite{Gepner-enriched-quasi}\\
                 &  dg-quasi-categories \cite{Arne-dg-quasi-cat} \\
\hline
Segal categories \cite{Dwyer-Kan-Smith} & Segal dg-categories \cite{Bacard-enriched-Segal-cat} \\
\hline
complete Segal spaces \cite{ComSegalSpacesREZK} & ??\\
\hline
\end{tabular}
\end{center}

There isn't, though, any enriched version of complete Segal spaces, and also none of the "models" we have presented here have been proven to be Quillen equivalent to dg-categories yet. These are the issues we intend to tackle in this article: indeed, we will define a concept of complete dg-Segal space and prove that the model category of complete dg-Segal spaces is Quillen equivalent to that of dg-categories with Tabuada's model structure. Unfortunately, that proof is conditional to a certain result (hypothesis \ref{Hyp: Re-zk}) that we have been unable to solve for now. We hope to make progress on this in future work.

\subsection{Motivation}

Why complete dg-Segal spaces, though? There are two main reasons.

On one hand, as in the case of complete Segal spaces, our definition of complete dg-Segal spaces will be done as a subcategory of a category of functors, in this case simplicial functors from free dg-categories of finite type into simplicial sets. It is well known that categories of functors are generally relatively well behaved: for example, the model structure that is so complicated to find in Mertens or Bacard's case will be deduced in a semi-straighforward manner from the one on simplicial sets in our case. But that is not all. To quote Dugger in \cite{Dugger}, a subcategory of a diagram category is "a kind of presentation by generators and relations", and having such a presentation makes the task of defining functors from dg-categories much easier. Indeed, instead of needing to define it over all dg-categories, it would now suffice to define it over free dg-categories of finite type and making sure the "relations" are sent to weak equivalences.

This approach to complete Segal spaces as a kind of presentation by generators and relations is not new: indeed, in 2005 Toën used the complete Segal space model for $\infty$-categories in \cite{TOEN-infini} to prove that the group of automorphisms on $\infty$-categories was $\Z/2\Z$. This result would be generalized in 2021 by Barwick and Schommer-Pries in \cite{infinity-n-cat}, where they proved that the group of automorphisms on $(\infty, n)$-categories is $(\Z/2\Z)^n$ for all $n\in\N^*$. In consequence, it is reasonable to expect that the construction of dg-categories as complete dg-Segal spaces would make it easier to compute the automorphisms of dg-categories.

On the other hand, in the last section we have pointed to the fact that dg-categories have a model structure and a monoidal structure, but that the two of them are not compatible. Indeed, the object $\D(1,0,1)$ (i.e., the dg-category with two objects and $k$ as the complex of morphisms between the two) is cofibrant in the model category of dg-categories, but it is easy to prove that $\D(1,0,1)\otimes\D(1,0,1)$ is not. This is a significant hurdle in the use of dg-categories, as important classic category theory theorems, like the Barr-Beck theorem, can not only not be proven, but cannot even be stated properly in the linear case. Unfortunately, our first and most obvious attempt to find a product that is compatible with the model structure in the complete dg-Segal case, the convolution product, gets thwarted by the fact that free dg-categories are not closed for the product of dg-categories; but there is still reason to believe that finding such a product will be easier in the context of complete dg-Segal spaces, as once again functor categories tend to be better behaved.

\subsection{Main results and definitions}

Let $k$ be a commutative ring. We denote by $\Ch$ the category of cochain complexes, by $c\Free$ the category of cofibrant free dg-categories of finite type, and by $\FreeS$ the full subcategory of the simplicial localization of dg-categories by the weak equivalences that consists only of the objects in $c\Free$. Then, in Theorem \ref{Th: construction Sing_w} we can define a chain of Quillen adjunctions of the form
$$\Re:\Fun^\mathbb{S} (\FreeS\op, \sS)\rightleftharpoons \cdots\rightleftharpoons\dg:\Sing,$$
that can be derived into a single adjunction
$$\Ho(\Fun^\mathbb{S} (\FreeS^{op}, \sS))\rightleftharpoons \Ho(\dg). $$

Our main objective in this paper will be to find the subcategory of $\Fun^\mathbb{S} (\FreeS\op, \sS)$ such that $\Sing$ is an equivalence; that subcategory will be given by our definition of a complete dg-Segal space. In order to get that one, though, we will first need to define dg-Segal spaces. We denote by $Gr(\Ch)$ the category of graphs on cochain complexes, i.e. a graph whose edges are cochain complexes.

\begin{defin}[\ref{Def: dg-Segal spaces}] Let $F\in\Fun^\S(\FreeS\op,\sS)$ be a simplicial functor from cofibrant free dg-categories of finite type to simplicial sets. We say that $F$ satisfies the \textbf{dg-Segal conditions} if:
\begin{enumerate}
	\item For all $L, K\in\FreeS$, $F(L\coprod K)\to F(L)\times F(K)$ is a weak equivalence. 
	\item The image of the initial object is a point, i.e. $F(\emptyset)\simeq *$.
	\item Let $G$ be a graph in $Gr(\Ch)$ and $x,y\in\Ob(G)$. For all $\alpha\in Z^n(G(x,y))$, the image of the free dg-category issued from $G[\cancel{\alpha}]$ is a homotopy pullback in $\sS$ of the following form:
	\begin{center}
		\begin{tikzcd}
			F(L(G[\cancel{\alpha}]))\ar[r]\ar[d]\arrow[dr,phantom, "\ulcorner^h", very near start]& F(L(G))\ar[d]\\
			F(\D^c(1,n,1))\ar[r]& F(\D(1,n,1))
		\end{tikzcd}
	\end{center}	
	where $L$ is the free functor from graphs to dg-categories, $\D^c(1,n,1)=L(k^c[n])$ and $\D(1,n,1)=L(k[n])$.
\end{enumerate}
	We denote the full subcategory of $F\in\Fun^\S(\FreeS\op,\sS)$ that satisfy the dg-Segal conditions by $\mathbf{dg-Segal}$ and call its objects \textbf{dg-Segal spaces}.
\end{defin}

But we cannot understand that definition on its own: what is $G[\cancel{\alpha}]$? Intuitively, if $G$ is a graph over cochain complexes and $\alpha$ is a cocycle in degree $n$, $G[\cancel{\alpha}]$ is a graph with the same objects and the same morphisms, except that we add an element $\beta$ such that $d(\beta)=\alpha$.

\begin{defin}[\ref{Def: G alpha}]Let $G\in Gr(\Ch)$ be a graph on the category of complexes, $x,y\in \Ob(G)$ two objects in $G$, and $\alpha\in Z^n(G(x,y))$ a cycle in $G(x,y)$. We define the graph $G[\cancel{\alpha}]$ to be a pushout in the category of graphs over the morphism $k[n]\to k^c[n]$, where $k[n]$ is the graph with two objects and the morphism complex between the two is a complex concentrated in degree $n$, where it is $k$; and $k^c[n]$ is the graph with two objects and the complex between the two which is always zero except for the degrees $n-1$ and $n$, where it is $k$.
\begin{center}
	\begin{tikzcd}
		k[n]\ar[r, "\alpha"]\ar[d]& G\ar[d]\\
		k^c[n]\ar[r]& G[\cancel{\alpha}].
	\end{tikzcd}
\end{center}
\end{defin}

We can prove that every object in the image of $\Sing$ is a dg-Segal space; and we can also construct a model structure for $\Fun^\S(\FreeS\op,\sS)$ where the dg-Segal spaces are its fibrant objects. But as it happened in the case of classical Segal spaces and $\infty$-categories, this object is not enough to completely determine dg-categories. As such, we will define complete dg-Segal spaces. For those, we will use the classical definition of a complete Segal space. 

\begin{prop}[\ref{Prop: delinearisation}] There exists a morphism, called \textbf{the linearisation of $\Delta$}, between the categories $\Delta$ and $\FreeS$, and it defines a Quillen adjunction between the categories $\Fun^\S(\FreeS\op,\sS)$ and $\Fun(\Delta\op, \sS)$ with their respective projective structures. 

  We call the morphism $j^*: \Fun^\S(\FreeS\op,\sS)\to \Fun^\S(\Delta\op, \sS)$ the \textbf{delinearisation morphism}.
\end{prop}

As we now have a direct way of comparing dg-Segal spaces and classical Segal spaces, we can prove that the image by the delinearisation functor of a dg-Segal space is always a Segal space; and using that, we will define the complete dg-Segal spaces as being the dg-Segal spaces such that their delinearised version is a complete Segal space.

\begin{defin}[\ref{Def: complete dg-Segal space}]Let $F$ be an object in $\Fun^\S(\FreeS\op,\sS)$. We say that $F$ is \textbf{a complete dg-Segal space} if it is a dg-Segal space and the morphism
$$F(k)\to F_{hoequiv} $$
is a weak equivalence, where $F_{hoequiv}$ is the subset of $F(\Delta_k^1)$ whose $0$-simplexes are homotopy equivalences. 
\end{defin}

As in the case of dg-Segal spaces, we can prove that the image of $\Sing$ is also contained in complete dg-Segal spaces, and also we can define a model structure such that the complete dg-Segal spaces are its fibrant objects.


\begin{teo}[\ref{Th: complee dg-Segal model structure}]There exists a class of morphisms $C$ and a simplicial closed model structure on $\Fun^\S(\FreeS\op,\sS)$ such that
\begin{enumerate}
	\item The cofibrations are the same as in the projective model structure.
	\item The fibrant objects are the complete dg-Segal spaces that are fibrant for the projective model structure.
	\item The weak equivalences are the $C$-local equivalences.
\end{enumerate}
We call such a model structure \textbf{the complete dg-Segal model structure}, and we denote it by $\dgSc$.
\end{teo}

Unfortunately for us, this construction gives us the weak equivalences in terms of $C$-local equivalences only: in order to prove the equivalence we will need an explicit definition of such a morphism, which we will call a DK-equivalence. That description turned out to be more complicated than expected, and it is left here as a hypothesis: suffice here to say that in the case where $F$ and $G$ are both in the image of $\Sing$, then a DK-equivalence $F\to G$ is given by an equivalence of dg-categories. The simplicial version of this result is a well-known proof by Rezk in \cite{ComSegalSpacesREZK}, but unfortunately it is not easy to generalize, and we find obstacles that were not there in Rezk's case. We briefly discuss the problems and possible solutions of this in Section \ref{sec: hypothesis-solving}.

\begin{hyp}[\ref{Hyp: Re-zk}]Let $f:F\to G$ be a morphism between two functors satisfying the dg-Segal conditions. Then $f$ is a DK-equivalence if and only if it is a weak equivalence in the complete dg-Segal model structure.  
\end{hyp}

Accepting that hypothesis to be true, though, we can prove that complete dg-Segal spaces are in fact a model of dg-categories. Indeed, we have the following result:

\begin{teo}[\ref{Th: fully faithfulness}, \ref{Th: equivalence}]Assume Hypothesis \ref{Hyp: Re-zk} is true. Then, there exists an equivalence of categories of the form
$$\Ho(\dg)\to \Ho(\dgSc). $$
\end{teo}

Even assuming the hypothesis, the proof of the theorem is already an elaborate one. Indeed, it involves a new notion of hypercovers on model categories and another in the context of dg-Segal spaces that we think could be of independent interest.

\subsection{Plan}

This paper is divided into five sections.

In section $2$, we remind the reader of basic concepts around dg-categories and simplicial localizations, and we construct Dugger's Universal Model Category. Basic knowledge of model categories and enriched categories is required to follow it.

Section $3$ contains all the most important definitions. It starts with the construction of the adjunction between dg-categories and simplicial functors between free dg-categories of finite type and simplicial sets. A proof of the adjunction will be given. It then goes on to give the definition of dg-Segal spaces, complete dg-Segal spaces and DK-equivalences, as well as the proof of the model structure of both the dg-Segal spaces and the complete dg-Segal spaces.

Section $4$ is the most technical one, developing a theory of hypercovers in a model category, and in particular hypercovers of dg-categories by free dg-categories. These are the tools that will be needed in the next section to prove the equivalence of the homotopy categories. It can be safely skipped by a first-time reader who is interested in the result but not its proofs.

Section $5$ contains the proofs of the fully faithfullness and essential surjectivity of the functor Sing from dg-categories to dg-Segal spaces constructed in section $3$. 

Lastly, in section $6$ we will briefly discuss  possible avenues for solving the hypothesis, offer a possible improvement on the definition of complete dg-Segal spaces, and say a few words about the automorphisms of dg-categories.

\subsection{Acknowledgments}

I would like to warmly thank Bertrand Toën for supervising the thesis work that made this article possible, and helping me along the way, always pointing in the right direction and providing useful ideas and commentary. I also wanted to extend my thanks to Niels Feld, Wendy Lowen and Arne Mertens for reading through different versions of this text and making comments that helped make it significantly better written and structured. Lastly, I would want to emphatically thank Violeta Borges Marques and Arne Mertens for getting interested in my results and pushing our reflections far enough to make us realize an error in an early version of this paper.

\subsection{Notation}

\begin{itemize}
	\item We denote an adjunction between two categories by $F:M\rightleftharpoons N:G$, with the left adjoint always being the arrow on top.
	\item We denote the category of simplicial sets by $\sS$, the category of cochain complexes over $k$ by $\Ch$ and the category of dg-categories as $\dg$. 
	\item We fix $k$ to be a commutative ring. We denote the tensor product on $k$ in $\Ch$ as $-\otimes -$, and the shift on a cochain complex $A$ as $A[-]$.
	\item We denote by $k[s]$ the cochain complex concentrated in degree $s$, where it is $k$, and $k^c[s]$ to be the complex concentrated in degrees $s$ and $s-1$, where it is $k$.
	\item We denote by $\Fun(A,B)$ the category of functors between categories $A$ and $B$, and $\Fun^\S(A,B)$ the category of simplicial functors between simplicial categories $A$ and $B$.
	\item We denote by $Gr(A)$ the category of graphs enriched over a category $A$, and $Gr(\Ch)^{tf}$ the full subcategory of graphs enriched over complexes of finite type. 
	\item We denote by $\mathcal{L}$ the category of free dg-categories (see Definition \ref{Def: def free}), and $c\Free$ the full subcategory of cofibrant free dg-categories of finite type (see Definition \ref{Def: finite type}).
\end{itemize} 

\section{Background notions}

Let us start with some background we will be using in the next sections. We fix $k$ a commutative ring.

\subsection{Differential graded categories}

In this section, we will state some well-known results on dg-categories. In particular, we will define what a dg-category and a dg-functor are, thus forming the category of dg-categories, $\dg$. Then, we will define Tabuada's model structure on $\dg$ and its moniodal structure, and say a word about how these two structures relate to each other. 

For the purposes of dg-Segal spaces, though, dg-categories are still too big; we will need a smaller subcategory of generators, that we will get with free dg-categories of finite type. As such, we will define a graph on $\Ch$, we see how to construct a free dg-category from a graph, and add a condition on graphs to make them of finite type. 

Unless stated otherwise, the results from this section are taken from \cite{Toen-dg}, but \cite{Keller} is also  a very good reference for them.

\begin{defin}We define $T$ a \textbf{dg-category} (differential graded category) to be a category enriched over $\Ch$ the category of cochain complexes. Equivalently, a dg-category consists of the following data: 
\begin{itemize}
	\item A set of objects $\Ob(T)$.
	\item For every pair of objects in $T$, $(x,y)\in\Ob(T)^2$, a cochain complex $\Hom(x,y)\in \Ch$. 
	\item For every triple of objects in $T$, $(x,y,z)\in\Ob(T)^3$ a composition morphism in $\Ch$
	$$\mu: \Hom(x,y)\otimes\Hom(y,z)\to \Hom(x,z) $$
	with the usual associativity condition.
	\item For every object in $T$, $x\in T$, a morphism $e_x: k\to \Hom(x,x)$ that satisfies the usual unit condition with respect to the composition stated above, where $k$ is the dg-category with a single object and $k$ as its complex of morphisms.
\end{itemize}
\end{defin} 

\begin{defin} Let $T$ and $T'$ be two dg-categories. A \textbf{dg-functor} (also called a\textbf{ morphism of dg-categories}) is a functor $f: T\to T'$ enriched over the category of cochain complexes. Equivalently, it consists of the following data: 
\begin{itemize}
	\item A map of sets $\Ob(T)\to \Ob(T')$.
	\item For every pair of objects in $T$, $(x,y)\in \Ob(T)^2$, a morphism of cochain complexes
	$$\Hom(x,y)\to \Hom(f(x), f(y)). $$
	satisfying the usual associativity and unit conditions.
\end{itemize}  
\end{defin}

\begin{nota} We denote $\dg$ the category of dg-categories and dg-functors.
\end{nota}

For any dg-category $T$, we can define an associated category: the homotopy category of $T$. 

\begin{defin}Let $T$ be a dg-category. We call the \textbf{homotopy category of $T$}, and we denote it by $\l T\r$, a category which has the same objects as $T$ and whose morphisms are given by 
$$\l T \r(x,y)=H^0(\Hom_T(x,y))\ \ \forall (x,y)\in \Ob(T)^2, $$
i.e. the cohomology group of degree 0 of the complex of morphisms.

The composition in this category is given for all $(x,y,z)\in \Ob(T)^3$ by the composition of morphisms
$$H^0(\Hom_T(x,y))\otimes H^0(\Hom_T(y,z))\to H^0(\Hom_T(x,y)\otimes\Hom_T(y,z))\to H^0(\Hom_T(x,z)). $$
\end{defin}

It has been proven that $\dg$ has a model structure, and even a cofibrantly generated model structure. It is defined as follows: 

\begin{defin}Let $f:T\to T'$ be a morphism of dg-categories. 
\begin{itemize}
	\item We say that $f$ is \textbf{quasi-essentially surjective} if the induced morphism of homotopy categories, $\l f\r: \l T\r\to \l T'\r$ is essentially surjective.
	\item We say that $f$ is \textbf{quasi-fully faithful} if for any two objects $(x,y)\in \Ob(T)^2$ the corresponding morphism of complexes $T(x,y)\to T(f(x),f(y))$ is a weak equivalence of complexes.
	\item We say that $f$ is a \textbf{quasi-equivalence} if it is quasi-essentially surjective and quasi-fully faithful. 
\end{itemize}
\end{defin}

\begin{defin}\cite[Not. 2.5]{TAB} Let $T$ be a dg-category. We say that a morphism in $T$, $f\in Z^0(\Hom_T(x,y))$, is a \textbf{homotopy equivalence} if it becomes an isomorphism $H^0(f)$ in $\l T\r$. 
\end{defin}

\begin{defin}\cite[Def. 2.12]{TAB} Let $f:T\to T'$ be a morphism of dg-categories. We say that $f$ is a \textbf{fibration} if 
\begin{itemize}
	\item For every two objects $(x,y)\in \Ob(T)^2$, the corresponding morphism of complexes $T(x,y)\to T'(f(x),f(y))$ is a fibration of complexes, i.e. is surjective.
	\item For all $x\in \Ob(T)$,  $y'\in \Ob(T')$, and all $h:f(x)\to y'$ homotopy equivalences in $T'$, there exists an object $y\in \Ob(T)$ and a homotopy equivalence $g:x\to y$ in $T$ such that $f(y)=y'$ and $f(g)=h$. 
\end{itemize}
\end{defin}

\begin{nota}We denote by $k^c[s]$ the cochain complex concentrated in degrees $s$ and $s-1$, where it is $k$.
\end{nota}

\begin{teo}\label{ch. 1:Tab-model} (\cite[Def. 2.14]{TAB}, see \cite[Th. 2.1]{TAB_Fr} for a proof, in French) The category $\dg$ admits a model structure with the quasi-equivalences as weak equivalences and the fibrations as defined above. It is a cofibrantly generated model category, and the generating cofibrations $\{I, P(s)/\ s\in\Z\}$ are the following:
\begin{itemize}
	\item The functor $I$ is the unique dg-functor $\0\to k$, where $\0$ is the initial object in $\dg$. 
	\item For all $s\in\Z$, let $\D(1,s,1)$ be the dg-category with two objects, $0$ and $1$, where $\Hom(0,0)=\Hom(1,1)=k$, $\Hom(1,0)=0$ and $\Hom(0,1)=k[s]$; and let $\D^c(1,s,1)$ be the dg-category with two objects, $0$ and $1$, where $\Hom(0,0)=\Hom(1,1)=k$, $\Hom(1,0)=0$ and $\Hom(0,1)=k^c\l s\r$. The $P(s):\D(1,s,1)\to\D^c(1,s,1)$ are, for all $s\in\Z$, the dg-functors that send $0$ to $0$, $1$ to $1$, and $\Hom_{\D(1,s,1)}(0,1)$ to $\Hom_{\D^c(1,s,1)}(0,1)$ by the following morphism:
	\begin{center}
		\begin{tikzcd}
			\cdots\ar[r]&0\ar[r]\ar[d]&0\ar[r]\ar[d]& k\ar[r]\ar[d, "id"]& 0\ar[r]\ar[d]&\cdots\\
			\cdots\ar[r]&0\ar[r]      &k\ar[r]      &k\ar[r]             &0\ar[r]       &\cdots
		\end{tikzcd}
	\end{center}
\end{itemize}
\end{teo}

\begin{coro}\label{ch. 1:hom-cofib}With the model structure described above, all dg-categories are fibrant. Also, if a dg-category $T$ is cofibrant, for all $(x,y)\in\Ob(T)^2$ the complex $\Hom_T(x,y)$ is cofibrant for the model structure on $\Ch$.
\end{coro}

\begin{prop}\cite[Th. 1.10]{general-model-cat} With the model structure described above, the model category $\dg$ is right proper.
\end{prop}

We know that $\Ch$ has a tensor product, that we have denoted $-\otimes -$: using that, we can easily define a tensor product over the whole category $\dg$. 

\begin{defin}Let $T$ and $T'$ be two dg-categories. We define the tensor product of $T$ and $T'$ as a category $T\otimes T'$ such that
\begin{itemize}
	\item The objects in $T\otimes T'$ are the objects in $\Ob(T)\times \Ob(T')$. 
	\item For every pair of objects $(x,y), (x',y')$ in $T\otimes T'$, we have a cochain complex of the form
	$$T'\otimes T'((x,y), (x',y'))=T(x,x')\otimes T'(y,y'). $$
\end{itemize}
\end{defin}

\begin{prop}The tensor product defined above gives $\dg$ a closed symmetric monoidal structure. The unit for the monoidal structure is the dg-category with one object and $k$ as its complex of morphisms, and we will denote it by $\D(0)$ or $k$.
\end{prop}

\begin{rem}This symmetric monoidal structure is unfortunately not compatible with the model structure, which means we don't have a symmetric monoidal model category: indeed, the tensor product of two cofibrant objects isn't necessarily cofibrant itself. The tensor product does have enough good properties, though, so that it can be derived. 
\end{rem}

And now, to finish this section, we will give a couple of definitions for subcategories which will be useful later.

\begin{defin}We define $G$ a \textbf{graph on $\Ch$} to be a graph enriched over $\Ch$ the category of cochain complexes. Equivalently, it consists of the following data:
\begin{itemize}
	\item A set of objects $\Ob(G)$.
	\item For every pair of objects in $G$, $(x,y)\in\Ob(G)^2$, a cochain complex.
\end{itemize}
We denote the category of graphs on $\Ch$ by $Gr(\Ch)$.
\end{defin}

\begin{defin}\label{Def: def free}There exists a Quillen adjunction $L:Gr(\Ch) \rightleftharpoons \dg:U$ where $U$ is the forgetful functor. We call a \textbf{free dg-category} $T$ a dg-category such that there exists a $T'\in Gr(\Ch)$ with $T=L(T')$. We denote the full subcategory of free dg-categories by $\mathcal{L}$.

In particular, if $X$ is a graph on $\Ch$, then we have that $L(X)$ has the same objects as $X$ and that for all objects $x,y\in \Ob(X)=\mathcal{O}$,
$$L(X)(x,y)=\bigoplus_{m\in\N}\bigoplus_{(x_1,\ldots, x_m)\in\mathcal{O}^m}(X(x,x_1)\otimes \cdots \otimes X(x_m,y)). $$
\end{defin}

\begin{defin}\label{Def: finite type}Let $G\in Gr(\Ch)$ be a graph on $\Ch$. We say that $G$ is \textbf{a graph of finite type} if it has a finite number of vertices and the edges between two vertices are always perfect complexes. We denote the full subcategory of graphs of finite type by $Gr(\Ch)^{tf}$.
\end{defin}

\begin{defin}Let $L=L(G)$ be a free dg-category. We say that $L$ is a \textbf{free dg-category of finite type} if the underlying graph $G$ is a graph of finite type.
\end{defin}

\subsection{Simplicial localizations}

In this section, we will define simplicial localizations, i.e. simplicial categories that in some sense localize a category. Of course, it could happen that such a simplicial category was always trivial, in which case we would have gained nothing; but it is luckily not the case. There is a close relationship between them and the Gabriel-Zisman localization, which we state. We will define simplicial localizations for a general model category $M$; in the case where $M$ is the category of simplicial sets $\sS$, there is an equivalent construction that we denote by $\sS^{C,W}$. That construction will allow us to construct and state a fully faithful functor that we will use in later chapters.

All results in this section come from \cite[2.2 and 2.3]{HAG1} unless stated otherwise.

\begin{defin}\label{Ch. 1: def localization} Let $C$ be a category and $W$ a subset of its morphisms. We call a \textbf{simplicial localization of $C$ with respect to $W$} a pair $(L_WC, l)$ where 
\begin{itemize}
	\item $L_WC$ is a simplicial category
	\item $l:C\to L_WC$ is a morphism of simplicial categories (considering $C$ a simplicial category via the natural inclusion $\Set\to \sS$), called \textbf{the localization morphism},
\end{itemize} 
 such that for every simplicial category $T$, $l$ induces a equivalence of simplicial categories
$$l^*:\R\Fun^\S(L_WC,T)\simeq\R\Fun^\S_W(C,T) $$
where $\R \Fun^\S(L_W, T)$ is seen as an object of $\Ho(\sC)$, and $\R\Fun^\S_W(C,T)$ denotes the full subcategory of $\R\Fun^\S(C,T)$ in $\Ho(\sC)$ consisting of all simplicial morphisms that send $W$ to equivalences in $T$. In other words, the localization morphism is such that, for every simplicial category $T$, $l$ induces a morphism of simplicial categories
$$l^*:\R\Fun^\S(L_WC,T)\to\R\Fun^\S(C,T) $$
which is fully faithful and whose essential image consists of the functors sending the morphisms in $W$ to equivalences in $T$.
\end{defin}

\begin{rem}It has been proven that a simplicial localization as defined above always exists, and that it is equivalent to the Dwyer-Kan simplicial localization from \cite{DK-local}. 
\end{rem}

\begin{prop}\label{ch.1: simplicial loc. and pi_0}Let $C$ be a category, $W$ a subset of its morphisms. We take the simplicial localization of $C$ with respect to $W$, $(L_WC,l)$. The localization morphism induces an equivalence between the Gabriel-Zisman localization $C\l  W^{-1}\r$ and the homotopy category of $L_WC$, $\pi_0(L_WC)$.
\end{prop}

In the case where the category $C$ is a model category, we have a useful result on top of this. We will take the result not from \cite{HAG1}, but from \cite[Localization and model categories]{DAG}.

\begin{prop} Let $M$ be a model category, and $C$ a small category. Then there exists a natural equivalence of simplicial categories 
$$L_{W_C}(\Fun(C,M))\simeq \R\Fun^\S(C,L_WM), $$
where $W$ are the weak equivalences in $M$ and $W_C$ are the weak equivalences on $\Fun(C,M)$ using the projective model structure.
\end{prop}

\begin{rem}One consequence of this result is that if $M$ is a model category, $L_WM$ has all limits and colimits, and those limits and colimits can be computed using homotopy limits and colimits.
\end{rem}

We can ask what happens when we take $M=\sS$. In that case, we have a Quillen equivalent construction. We take $\sS$ as a simplicial model category where for all $x,y\in \sS$ we take $\Hom_n(x,y)=\Hom(x\times \Delta_n,y)$.

\begin{defin}Let $C$ be a simplicial category and $W$ a subset of its morphisms. We call the \textbf{model category of restricted diagrams from  $(C,W)$ to $\sS$} the left Bousfield localization of $\Fun^\S(C,\sS)$ along the set of morphisms of the form $\underline{h}_x\to \underline{h}_y$ for all $x\to y\in W$, where $\underline{h}_x:C\op\to \sS$ is given by $\underline{h}_x(z)=\Hom_C(y,z)$ as simplicial sets.
 We denote it by $\sS^{C,W}$.
\end{defin}

\begin{rem}By the general theory of Bousfield localizations, the fibrant objects of $\sS^{C,W}$ are those functors $f:C\to \sS$ that satisfy the following conditions:
\begin{itemize}
	\item The functor $f$ is a fibrant object for the projective model structure on $\Fun^\S(C,\sS)$.
	\item For all $x\to y$ morphism in $W$, the induced morphism $f(x)\to f(y)$ is an equivalence in $\sS$.
\end{itemize}
\end{rem}

\begin{teo}\label{Th: equivalence for functor categories}Let $C$ be a simplicial category, $W$ a subset of its morphisms. Let $(F_*C, F_*W)$ be the canonical free resolution of $(C,W)$ as simplicial categories. There exist two natural functors 
$$(C,W)\xleftarrow{p} (F_*C,F_*W)\xrightarrow{l} L_{F_*W}F_*C=L_WC $$
which induce two right Quillen functors
$$\sS^{C,W}\xrightarrow{p^*}\sS^{F_*C,F_*W}\xleftarrow{l^*}\Fun^\S(L_WC,\sS).$$
Those Quillen functors $p^*, l^*$ are Quillen equivalences. In particular, there exists a chain of Quillen equivalences between $\sS^{C,W}$ and $\Fun^\S(L_WC,\sS)$.
\end{teo}

\begin{rem}We bring to your attention that we have said "there exists a chain of Quillen equivalences" and not "there exists a Quillen equivalence". Indeed, we cannot compose $p_*$ and the left Quillen adjoint of $l_*$ and get a Quillen adjunction, as they aren't both right adjoints.
\end{rem}

\begin{coro}Let $C, D$ be two simplicial categories, $W$ a subset of morphisms of $C$ and $V$ a subset of morphisms of $D$. Let $f:C\to D$ be a morphism of simplicial categories such that $f(W)\subset V$. If the induced functor $Lf:L_WC\to L_VD$ is an equivalence of simplicial categories, then the Quillen adjunction 
$$f_!:\sS^{C,W}\rightleftharpoons \sS^{D,V}:f^*$$
 is a Quillen equivalence.
\end{coro}

As usual we can construct a Yoneda lemma for this type of localization too. This result comes from \cite[Theorem 4.2.3]{HAG1}. It is given there for a pseudo-model category.

\begin{teo}\label{Th: fully-faith}Let $C$ a model category that is also a simplicial category and $W$ the set of weak equivalences, then the functor 
$$\R\Sing(-):\Ho(C)\to \Ho(\sS^{C,W}) $$
is fully faithful. 
\end{teo}

\subsection{The universal model category}

Lastly in this section, we will construct a "universal model category" $UC$ for every category $C$, in the sense that for all model category $M$ and all functor $\gamma:C\to M$ there exists a factorization of $\gamma$ by $UC$ which is, in some sense, unique. From here on all results come from \cite{Dugger}.

\begin{defin}Let $C$ be a category, let $M$ and $N$ be two model categories. We fix a functor $f:C\to M$. For all $g:C\to N$, we define a \textbf{factorization of $g$ through $M$} to be a triple $(L,R,\eta)$ such that 
\begin{itemize}
	\item The functors $L:M\rightleftharpoons N:R$ form a Quillen adjunction.
	\item We have a weak equivalence $\eta:L\circ f\simeq g$ in the projective model structure on $\Fun(C,N)$.
\end{itemize}
\end{defin}

\begin{defin}Let $C$ be a category, let $M$ and $N$ be two model categories. We fix a functor $f:C\to M$, and take $g:C\to N$. If we have $(L,R,\eta)$ and $(L',R',\eta')$ two factorizations of $g$ through $M$, we define a \textbf{morphism of factorizations} to be a natural transformation $F:L\to L'$ such that for all $x\in C$ the diagram
\begin{center}
	\begin{tikzcd}
		L\circ f(x)\ar[rr,"F\circ \Id_f"]\ar[rd,"\eta"]& &L'\circ f(x)\ar[ld,"\eta'"]\\
		 & g(x) &
	\end{tikzcd}
\end{center}
commutes.
\end{defin}

\begin{nota}With the above conditions, we denote $\Fact_f(g)$ the category of factorizations of $g$ through $M$ and morphisms between them.
\end{nota}

\begin{prop}\label{ch. 1: Dugger}\cite[Prop. 2.3]{Dugger} Let $C$ be a category and $M$ be a model category. There exists a closed model category $UC$ and a functor $r:C\to UC$ such that the following is true: for every functor $\gamma:C\to M$ there exists a factorization of $\gamma$ through $UC$, $(\Re,\Sing,\eta)$, 
\begin{center}
	\begin{tikzcd}
		C\ar[rr,"r"]\ar[rrdd,"\gamma" {name=A}]&&UC\ar[dd,"\Re"', shift right]\ar[Rightarrow, to=A, "\eta" above]\\
		&  &\\
		 &&M\ar[uu, "\Sing"', shift right]
	\end{tikzcd}
\end{center}
and the category of factorizations $\Fact_r(\gamma)$ is contractible.
\end{prop}

\begin{proof}[Sketch of construction]\cite[9.5 Section 3]{Dugger}
The universal model category $UC$ is no other than $UC=\Fun(C\op,\sS)$, the category of simplicial presheaves over $C$. The inclusion $i:\Set\to \sS$ induces an obvious inclusion $j:\Fun(C\op,\Set)\to \Fun(C\op,\sS)$,  which composed with the Yoneda embedding gives us the needed functor from $C$ to $\Fun(C\op,\sS)$, $r=j\circ h:C\to \Fun(C\op,\Set)\to \Fun(C\op, \sS)$. 

Now that we have the closed model category $\Fun(C\op, \sS)$ and the functor $r:C\to \Fun(C\op, \sS)$, we take a functor $\gamma:C\to M$ to another model category. We are going to give the factorization, but we won't prove that it is in fact one. The easier adjoint to define is $\Sing$. We define $\Sing$ as follows:
$$\Sing(x)=\Map( \gamma(-),x)\in \Fun(C\op, \sS). $$
The functor $\Re$ is just the left adjoint of $\Sing$.
\end{proof}

\section{The Sing functor and dg-Segal spaces }\label{sec: section 1}

Now that we have the necessary background notions, we can start on our results. In this chapter, we will first construct the necessary adjunction between dg-categories and our new functor category; in the next section we will go on to define dg-Segal spaces and their model structure, and then give the relationship between dg-Segal spaces and classical Segal spaces. Then in a third section we will use that relationship to define complete dg-Segal spaces and their model structure. Lastly, we will define DK-equivalences and end by stating the hypothesis we will use in the last section.

\subsection{Constructing the adjunction}

Let us work on defining the adjunction between dg-categories and the functor category that will function as our ambient category for our complete dg-Segal spaces. To be precise, we will not get an adjunction between the model categories: instead, we will construct a \textit{chain} of Quillen adjunctions from $\dg$ to $\Fun(\FreeS\op,\sS)$. But we cannot compose right Quillen functors to left Quillen functors and still get a Quillen adjunction. And in this case, one of the terms of the chain goes in the "wrong direction", which means this chain cannot be composed at the level of model categories. There is a way to avoid this issue on the level of homotopy categories, though.

\begin{nota}We denote the full subcategory of cofibrant free dg-categories of finite type by $c\Free\subset\mathcal{L}$. As we won't consider any other type in this text, we will most of the time omit the term "cofibrant" from our explanations.
\end{nota}

\begin{defin}Let $W$ be the set of weak equivalences in Tabuada's model structure on $\dg$ the category of dg-categories. We construct the simplicial localization $L_W\dg$ of dg-categories with respect to $W$. We define the \textbf{simplicial $c\Free$}, and we denote it by $\FreeS$, the full simplicial subcategory of $L_W\dg$ whose objects are the ones in $c\Free$.
\end{defin}

\begin{rem}We must be careful with the definition of $\FreeS$. It is tempting to just define it as $L_Wc\Free$, but the two categories $L_Wc\Free$ and $\FreeS$ do not coincide. 
\end{rem}

\begin{teo}\label{Th: construction Sing_w}There exists a chain of Quillen adjunctions of the form 
$$\Re:\Fun^\mathbb{S} (\FreeS\op, \sS)\rightleftharpoons \cdots\rightleftharpoons\dg:\Sing,$$
and it can be derived into a single adjunction
$$\Ho(\Fun^\mathbb{S} (\FreeS^{op}, \sS))\rightleftharpoons \Ho(\dg). $$
\end{teo}

\begin{proof}
Let us start from the right. As we have defined $\FreeS$ to have the simplicial structure induced by the simplicial structure of $L_W\dg$, there exists a natural simplicial inclusion functor $j:\FreeS\hookrightarrow L_W\dg$. If we take the projective model structure on categories of simplicial functors, we have a Quillen adjunction of the form 
$$j_!:\Fun^\S(\FreeS\op, \sS)\rightleftharpoons \Fun^\S(L_W\dg\op,\sS):j^* $$
Now, by definition of the simplicial localization, there exists a functor $l:\dg\to L_W\dg$ that gives us a Quillen adjunction of the form 
$$l_!:\Fun(\dg\op, \sS)\rightleftharpoons \Fun^\S(L_W\dg\op, \sS):l^* $$
And lastly, using Dugger's Universal Model Category from \cite{Dugger}, if we take both the category $C$ and the model category $M$ to be $\dg$, and the functor $\gamma$ to be the identity, we have a factorization $(\Re_W, \Sing_W,\eta)$,
$$\Re_W:\Fun(\dg\op, \sS)\rightleftharpoons \dg:\Sing_W $$
where we know that the right adjoint is given by $\Sing_W(X)=\Map(-,X)$.

 These Quillen adjunctions cannot be composed. Indeed, if we write the chain of Quillen adjunctions and we write the left Quillen functor always on top, we get the following diagram:
\begin{center}
	\begin{tikzcd}
		\Fun^\S(\FreeS\op,\sS)\ar[r, "j_!", shift left]& \ar[l, "j^*", shift left]\Fun^\S(L_W\dg\op,\sS)\ar[r,  "l^*" below, shift right]& \ar[l, "l_!" above, shift right]\Fun(\dg\op,\sS)						\ar[r,"\Re_W", shift left]& \ar[l,"\Sing_W", shift left] \dg.
	\end{tikzcd}
\end{center}
At this level there is nothing we can do to fix this: but on the homotopy categories there is. Indeed, in the homotopy categories we can construct a factorization of the functor $\Sing_W$, which would bypass the need for the adjoint $l_!$ altogether. In other words, we are going to try and find a functor $\Sing'$ such that the diagram 
\begin{center}
	\begin{tikzcd}
		\Ho(\Fun^\S(\FreeS\op,\sS))\ar[r, "j_!", shift left]& \ar[l, "j^*", shift left]\Ho(\Fun^\S(L_W\dg\op,\sS))\ar[r,  "l^*" below, shift right]& \ar[l, "l_!" above, shift right]\Ho(\Fun(\dg\op,\sS))						\ar[d,"\Re_W", shift left]\\
		&& \ar[u,"\Sing_W", shift left] \Ho(\dg)\ar[ul, "\Sing'", dashed].
	\end{tikzcd}
\end{center}
 commutes. 

By definition of the simplicial localization category, if we have a functor $F:\dg\to \sS$ that sends all morphisms in $W$ to weak equivalences in $\sS$, then it can be factorized through $l^*$ on the homotopy categories. But $\Sing_W(X)=\Map(-,X)$ is a right Quillen adjoint, so it sends weak equivalences between fibrant objects to weak equivalences, and all objects in $\dg$ are fibrant. So for all $X\in\dg$, the image $\Sing_W(X)$ can be factorized through $L_W\dg$. That gives us a functorial factorization of the form
$$\Sing_W(X)=l^*\circ \Sing'(X) $$
in the homotopy categories.

Now we only need to prove that the functors $\Sing=j^*\circ\Sing'$ and $\Re=\Re_W\circ l^*\circ j_!$ are truly adjoints, i.e. that for all $X\in\Fun^\S(\FreeS\op,\sS)$ and for all $Y\in\dg$ there exists an functorial isomorphism of the form
$$\l X,j^*\circ \Sing'(Y) \r_{\Fun^\S(\FreeS\op,\sS)}\simeq \l\Re_W\circ l^*\circ j_!(X),Y\r_{\dg}. $$
Let us start on the left side and work our way through. We start with $\l X,j^*\circ \Sing'(Y) \r_{\Fun^\S(\FreeS\op,\sS)}$. As the functors $j_!$ and $j^*$ are adjoints, we have that 
$$\l X,j^*\circ \Sing'(Y) \r_{\Fun^\S(\FreeS\op,\sS)}\simeq \l j_!(X),\Sing'(Y)\r_{\Fun^\S(L_W\dg\op,\sS)}. $$
Because we are working on the homotopy categories, by the definition of simplicial localizations the functor $l^*$ is fully faithful. That means, in particular, that we get the following isomorphism:
$$\l j_!(X),\Sing'(Y)\r_{\Fun^\S(L_W\dg\op,\sS)}\simeq  \l l^*\circ j_!(X),\Sing_W(Y)\r_{\Fun(\dg\op,\sS)}.$$
And finally, using the definition of an adjunction again on $\Re_W$ and $\Sing_W$, we get that
$$\l l^*\circ j_!(X),\Sing_W(Y)\r_{\Fun(\dg\op,\sS)}\simeq \l\Re_W\circ l^*\circ j_!(X),Y \r_\dg. $$
We have the isomorphism
$$\l X,\Sing(Y) \r_{\Fun^\S(\FreeS\op,\sS)}\simeq \l\Re(X),Y \r_{\dg}$$
and the pair $\Re:\Ho(\Fun^\S(\FreeS\op,\sS))\rightleftharpoons \Ho(\dg):\Sing$ is an adjunction on the homotopy categories. We have finished the proof.
\end{proof}

We have now an adjunction between the categories we wanted it for. The next step is proving that this functor is an equivalence.  We have to define the image of $\Sing$ first, and that will take some work.

\subsection{dg-Segal spaces and the delinearization functor}

In this section, we will give a definition of dg-Segal spaces, and prove that every element in the image of the $\Sing$ functor is a dg-Segal space. Next, we will define a model structure for them. In \cite{ComSegalSpacesREZK}, Rezk takes his model structure over $\Fun(\Delta\op,\sS)$ and does a Bousfield localization that makes the Segal spaces into its fibrant objects; he calls that \textbf{the Segal space model category structure}. Following his footsteps, we will get a new model structure for $\Fun^\S(\FreeS\op,\sS)$ where the fibrant objects are the dg-Segal spaces that are fibrant for the projective structure. 

Considering how we have followed Rezk's method pretty closely, it won't be surprising to our readers to see that there is a close relationship between our dg-Segal spaces and classic Segal spaces. Indeed, we will define a Quillen adjunction between $\Fun^\S(\FreeS\op,\sS)$ and the model category of simplicial spaces, called the delinearization functor. We will then go on to prove that the image of every dg-Segal space by the delinearization functor is a Segal space in the classical sense.

On top of our own results, we will also recall some results on the classical case to the reader as they are used. Even though the results about Segal and complete Segal spaces in this section are due to Rezk, we will take inspiration in Rasekh's lecture notes in \cite{Nima} for their presentation.

 As we said in the introduction, we take our inspiration for this section from complete Segal spaces. We remind the reader that the definition of said spaces is the following.

\begin{defin}\cite[Def. 4.1]{ComSegalSpacesREZK} Let $W$ be a Reedy fibrant simplicial space. We say that $W$ is a \textbf{Segal space} if the maps 
$$W_k\to \overbrace{W_1\times_{W_0}\cdots\times_{W_0}W_1}^{\text{$k$ times}}$$
are weak equivalences for all $k\geq 2$.
\end{defin}

Here, though, we aren't just working with simplicial sets: we have a linear structure to worry about. Consequently, we are going to define the action of adding a module to a complex of modules.

\begin{defin}\label{Def: G alpha}Let $G\in Gr(\Ch)$ be a graph in the category of complexes, $x,y\in \Ob(G)$ two objects in $G$, and $\alpha\in Z^n(G(x,y))$ a cycle in $G(x,y)$. We define the graph $G[\cancel{\alpha}]$ to be a graph of complexes such that 
\begin{itemize}
	\item The graph $G[\cancel{\alpha}]$ has the same objects as $G$.
	\item The graph $G[\cancel{\alpha}]$ has the same morphisms as $G$ between $x',y'\in\Ob(G)$ if $(x',y')\neq (x,y)$, i.e. $G[\cancel{\alpha}](x',y')=G(x',y')$.
	\item We define $G[\cancel{\alpha}](x,y)$ to be the complex of morphisms $G(x,y)\oplus_k \beta$ where $d\beta=\alpha$.
\end{itemize}
In other words, we have that $G[\cancel{\alpha}]$ is a pushout in the graphs over the morphism $k[n]\to k^c[n]$, where $k[n]$ is the graph with two objects, $0,1$, and $k[n]$ as $\Hom(0,1)$ the morphism between the two, and $k^c[n]$ is the graph with two objects, $0,1$, and the complex $\Hom(0,1)$ which is always zero except for the degrees $n-1$ and $n$, where it is $k$.
\begin{center}
	\begin{tikzcd}
		k[n]\ar[r, "\alpha"]\ar[d]& G\ar[d]\\
		k^c[n]\ar[r]& G[\cancel{\alpha}].
	\end{tikzcd}
\end{center}

\end{defin}

\begin{rem}It isn't hard to see that all we have done here has been adding a term in degree $n-1$ to the complex of modules $G(x,y)$, effectively killing $\alpha$ in the homotopy category.
\end{rem}

Now that we have this definition, we can apply it to finally define the conditions of our image.

\begin{defin}\label{Def: dg-Segal spaces} Let $F\in\Fun^\S(\FreeS\op,\sS)$ be a simplicial functor from cofibrant free dg-categories to simplicial sets. We say that $F$ satisfies the \textbf{dg-Segal conditions} if:
\begin{enumerate}
	\item For all $L, K\in\FreeS$, $F(L\coprod K)\to F(L)\times F(K)$ is a weak equivalence. 
	\item The image of the initial object is a point, i.e. $F(\emptyset)\simeq *$.
	\item Let $G$ be a graph in $Gr(\Ch)$ and $x,y\in\Ob(G)$. For all $\alpha\in Z^n(G(x,y))$, the image of the free dg-category issued from $G[\cancel{\alpha}]$ is a homotopy pullback in $\sS$ of the following form:
	\begin{center}
		\begin{tikzcd}
			F(L(G[\cancel{\alpha}]))\ar[r]\ar[d]\arrow[dr,phantom, "\ulcorner^h", very near start]& F(L(G))\ar[d]\\
			F(\D^c(1,n,1))\ar[r]& F(\D(1,n,1))
		\end{tikzcd}
	\end{center}	
	where $\D^c(1,n,1)=L(k^c[n])$ and $\D(1,n,1)=L(k[n])$; i.e. $F$ sends the homotopy pushouts of the previous definition to homotopy pullbacks.
\end{enumerate}
	We denote the full subcategory of $F\in\Fun^\S(\FreeS\op,\sS)$ that satisfies the dg-Segal conditions by $\mathbf{dg-Segal}$ and call its objects \textbf{dg-Segal spaces}.
\end{defin}

\begin{rem}We will see in Sublemma \ref{SLem: equivalent condition on dg-Segal} that we could probably give an equivalent definition of dg-Segal spaces which is much closer to the classical definition of a Segal space. We have decided not to do so because we think this definition emphasizes better the linear sturcture. In any case, this is not a minimal construction: see Section \ref{sec: Delta_k} for a suggestion of a minimal construction for dg-Segal spaces.
\end{rem}

We have constructed this under the assumption that the image of the functor $\Sing$ we defined in Theorem \ref{Th: construction Sing_w} is formed up to weak equivalence of the functors that satisfy the dg-Segal conditions. In order to prove that, first we need to prove that every object in the image is of this form.

\begin{prop}\label{Prop: Sing is dg-Segal}Let $T\in \dg$ be a dg-category. Then the functor $\Sing(T)$ satisfies the dg-Segal conditions.
\end{prop}

\begin{proof}
 We have to see that $T$ fulfills the three conditions of the definition.  

1. Let $L,K\in\FreeS$ be two cofibrant free dg-categories. As $L,K$ are cofibrant, we have that 
$$\Sing(T)(L\coprod K)=\Map(L\coprod K,T)=\Hom(L\coprod K, C_*(T)).$$
But by definition of coproduct, the condition 1. holds in this case:
$$\Hom(L\coprod K, C_*(T))=\Hom(L,C_*(T))\times \Hom(K,C_*(T))=\Map(L,T)\times \Map(K,T) $$
and we have finished.

2. This condition is obvious: $\Sing(\emptyset)=\Map(-,\emptyset)=*$ by definition of final object.

3. We need to prove that the following diagram is a homotopy pullback
\begin{center}
		\begin{tikzcd}
			\Sing(T)(L(G[\cancel{\alpha}]))=\Map(L(G[\cancel{\alpha}]), T)\ar[r]\ar[d]& \Sing(T)(L(G))=\Map(L(G),T)\ar[d]\\
			\Sing(T)(\D^c(1,n,1))=\Map(\D^c(1,n,1), T)\ar[r]& \Sing(T)(\D(1,n,1))=\Map(\D(1,n,1),T).
		\end{tikzcd}
	\end{center}
	
All these objects are in $\sS$, which is a proper category, which means that if one of these arrows is a fibration, then we have a homotopy pullback. By Tabuada's model structure, we have that $\D(1,s,1)\to \D^c(1,s,1)$ is a generating cofibration in $\dg$, and by construction $\Sing(T)(\D^c(1,s,1))\to \Sing(T)(\D(1,s,1))$ is a fibration. We then have that the previous diagram is a homotopy pullback.

The functor $\Sing(T)$ satisfies the dg-Segal conditions and we have finished our proof.
\end{proof}

\begin{rem}We draw the attention of our readers to the fact that, even though we haven't asked for dg-Segal spaces to be fibrant, by adjointness every $\Sing(T)$ is actually fibrant.
\end{rem}

We have proven that every element in the image of $\Sing$ is a dg-Segal space. We ask the reader to keep that in mind for when we have to prove the essential surjectivity. But for now, let us focus on the model structure we can get for these dg-Segal spaces. 

\begin{teo}There exists a simplicial closed model structure on $\Fun^\S(\FreeS\op,\sS)$ and a class of morphisms $C$ such that
\begin{enumerate}
	\item The cofibrations are the same as in the projective model structure.
	\item The fibrant objects are the dg-Segal spaces which are fibrant for the projective model structure.
	\item The weak equivalences are the $C$-local equivalences with respect to the class $C$.
\end{enumerate}
We call such a model structure \textbf{the dg-Segal model structure.}
\end{teo}

\begin{proof}
In order to prove this, we are going to use the left Bousfield localization. The first thing we need to do is find a class of morphisms $C$ such that the dg-Segal spaces are exactly the $C$-local objects, i.e. such that for every dg-Segal space $F$ and for every morphism $f:A\to B$ in $C$, the morphism 
$$\Map(B,F)\to\Map(A,F) $$
is a weak equivalence. For that, we define three classes of morphisms, one for each condition in the definition of a dg-Segal space.

1. Our first class of morphisms is $C_1=\{\Sing(L)\coprod\Sing(L')\to \Sing(L\coprod L')/\ L,L'\in \FreeS\}$. Let us check that the $C_1$-local objects are exactly the functors that satisfy the first condition of the definition of dg-Segal spaces. Let $F$ be a functor in $\Fun^\S(\FreeS\op,\sS)$ that is a $C_1$-local object. Then, by the Yoneda Lemma, we have that 
$$\Map(\Sing(L\coprod L'),F)\simeq F(L\coprod L')\to \Map(\Sing(L)\coprod\Sing(L'),F)\simeq F(L)\times F(L') $$
is a weak equivalence. By construction, $F$ satisfies the first dg-Segal condition. 

2. As the second condition of the definition of a dg-Segal space is one single weak equivalence, the class of morphisms associated to it will also have just one morphism. We consider $C_2=\{\emptyset\to \Sing(\emptyset\}$. Let $F$ be a $C_2$-local object. Then, also by the Yoneda lemma, we have that
$$\Map(\Sing(\emptyset),F)\simeq F(\emptyset)\to \Map(\emptyset,F)\simeq * $$
is a weak equivalence, and $F$ satisfies the second dg-Segal condition.

3. Lastly, we take the class 
$$C_3=\{\Sing(L(G))\coprod_{\Sing(\D(1,s,1))}\Sing(\D^c(1,s,1))\to \Sing(L(G[\cancel{\alpha}])/$$ $$G\in Gr(\Ch),s\in\Z, x,y\in\Ob(G), \alpha\in Z^n(x,y)\}.$$ Let $F$ be a $C_3$-local object. Then, by the Yoneda lemma, we have that the morphism 
$$\Map(\Sing(L(G[\cancel{\alpha}])),F)\simeq F(L(G[\cancel{\alpha}]))\to $$
$$\Map(\Sing(L(G))\coprod_{\Sing(\D(1,n,1))}\Sing(\D^c(1,n,1)), F)\simeq F(L(G))\times_{\Sing(\D(1,n,1))}F(\D^c(1,n,1)) $$
is a weak equivalence. Alternatively, that means that the diagram in condition 3 of the dg-Segal condition is a homotopy pullback and $F$ satisfies the third dg-Segal condition.

We take the class of morphisms $C=C_1\cup C_2\cup C_3$ to be our $C$ in the Bousfield localization. If such a localization exists, its fibrant objects will be exactly the dg-Segal spaces which are fibrant for the projective model structure.

We know that if the category we are trying to localize is a left proper cellular model category, then the left Bousfield localization exists. But the category of simplicial sets $\sS$ is left proper and cellular, and the functors on it are also left proper and cellular. So this localization exists and we have finished.
\end{proof}

Let us construct a Quillen adjoint between the model category $\Fun^\S(\FreeS\op,\sS)$ and the model category $\Fun(\Delta\op, \sS)$.

\begin{prop}\label{Prop: delinearisation}There exists a morphism, called \textbf{the linearisation of $\Delta$}, between the categories $\Delta$ and $\FreeS$, and it defines a Quillen adjunction between the categories $\Fun^\S(\FreeS\op,\sS)$ and $\Fun(\Delta\op, \sS)$ with their respective projective structures.
\end{prop}

\begin{proof}
Let $[n]\in\Delta$ be an object in $\Delta$. We define $j([n])=[n]\times k$ to be a free dg-category such that 
\begin{center}
	\begin{tikzcd}
		j([n])=0\ar[r,"k"]& 1\ar[r,"k"]&\cdots\ar[r,"k"]&n.
	\end{tikzcd}
\end{center}
This is a free dg-category of finite type, and it is also cofibrant, so this morphism $j$ is well-defined as $\Delta\to \FreeS$. We construct then the following Quillen adjunction:
$$j_!:\Fun(\Delta\op,\sS)\simeq \Fun^\S(\Delta\op, \sS)\rightleftharpoons \Fun^\S(\FreeS\op,\sS):j^* $$
and we have finished our proof.
\end{proof}

\begin{nota}
    We denote the images by the linearization morphism by $j(\l n\r)=\D^n$. 
\end{nota}

\begin{defin}
    We call the morphism $j^*: \Fun^\S(\FreeS\op,\sS)\to \Fun^\S(\Delta\op, \sS)$ the \textbf{delinearlisation morphism}.
\end{defin}

Now, we have calculated that Quillen adjunction for the projective model structure. But we have two localizations here: let us prove that this stays a Quillen adjunction in the localizations.

\begin{prop}\label{Prop: dg-Segal are Segal}Let $F\in\dgS$ be a dg-Segal space. Then its image by the delinearisation morphism $j^*$ is a Segal space.
\end{prop}

\begin{proof}
Let $F$ be a dg-Segal space. Then, in order to prove that its image by $j^*$ is a Segal space, by definition of Segal space we need to prove that for all $n\geq 1$, the morphism
$$j^*(F)_n\to  \overbrace{j^*(F)_1\times_{j^*(F)_0}\cdots\times_{j^*(F)_0}j^*(F)_1}^{\text{$n$ times}}$$
is a weak equivalence. If we unravel that definition, we have that for every $i\in\N$, $j^*(F)_i=j^*(F)([i])=F(j([i]))=F(\D^i)$. So proving that $j^*(F)$ is a Segal space can be rewritten as asking that for all $n\geq 1$,
$$\Phi_n:F(\D^n)\to F(\D^1)\times_{F(\D^0)}\cdots\times_{F(\D^0)}F(\D^1)$$
is a weak equivalence of simplicial spaces.

 We remark that $F(\D^0)=F(k)$.

We will prove the proposition by induction.

\begin{itemize}
	\item $n=1$. This is obvious, since $F(\D^1)\simeq F(\D^1)$. There is nothing to prove.
	\item $n\geq 2$. We assume that the morphism 
	$$\Phi_{n-1}: F(\D^{n-1})\to  \overbrace{F(\D^1)\times_{F(k)}\cdots\times_{F(k)}F(\D^1)}^{\text{$n-1$ times}} $$
	is a weak equivalence. Let us prove that the morphism $\Phi_n$ is also a weak equivalence.
	
	As usual for inductions, we have to decompose $F(\D^n)$ in a way that makes $F(\D^{n-1})$ appear. In this case, we will use the properties of a dg-Segal category to do so. We define $G^0$ to be a graph of the following form:
	\begin{center}
		\begin{tikzcd}
		G^0=j([n-1])\coprod *=0\ar[r,"k"]& 1\ar[r,"k"]&\cdots\ar[r,"k"]& n-1\ar[r,"0"]&n
		\end{tikzcd}
	\end{center}
	Then, we can construct $\D^n$ as the following pushout:
	\begin{center}
		\begin{tikzcd}
		\D(1,-1,1)\ar[r]\ar[d, "\alpha"]&\D^c(1,-1,1)\ar[d]\\
		L(G^0)\ar[r]& \D^n
		\end{tikzcd}
	\end{center}
	By using the third dg-Segal condition, we can write $F(\D^n)$ in the following way:
	$$F(\D^n)\simeq F(G^0)\times_{F(\D(1,-1,1))}F(\D^c(1,-1,1))\simeq F(\D^{n-1}\coprod k)\times_{F(\D(1,-1,1))}F(k\coprod k)). $$
	By the first dg-Segal condition, we can make those coproducts commute with $F$ in the following way:
	$$F(\D^n)\simeq F(\D^{n-1}\coprod k)\times_{F(\D(1,-1,1))}F(k\coprod k))\simeq (F(\D^{n-1})\times F(k))\times_{F(\D(1,-1,1))}(F(k)\times F(k)). $$
	Now, in particular, if we take $n=1$, we get the following formula: 
	$$F(\D^1)\simeq (F(\D^0)\times F(k))\times_{F(\D(1,-1,1))}F(k)^2=F(k)^2\times_{F(\D(1,-1,1))} F(k)^2.$$ 
	We are almost there. If we add and subtract one $F(k)$ to the formula of $F(\D^n)$ we will be done.
	$$F(\D^n)\simeq (F(\D^{n-1})\times F(k))\times_{F(\D(1,-1,1))}F(k)^2\simeq F(\D^{n-1})\times_{F(k)}(F(k)\times F(k))\times_{F(\D(1,-1,1))}F(k)^2 $$
	and we have that $F(\D^n)\simeq F(\D^{n-1})\times_{F(k)}F(\D^1)$. By the induction hypothesis, we have 
	$$F(\D^n)\simeq \overbrace{F(\D^1)\times_{F(\D^0)}\cdots\times_ {F(\D^0)}F(\D^1)}^{\text{$n-1$ times}}\times_{F(\D^0)} F(\D^1).$$
\end{itemize}
	The image $j^*(F)$ is a Segal space and we have finished our proof.
\end{proof}

\begin{coro}The adjunction $j_!:\Fun(\Delta\op,\sS)\rightleftharpoons \Fun^\S(\FreeS\op,\sS):j^*$ is a Quillen adjunction for the Segal and dg-Segal model structures, respectively.
\end{coro}

\begin{proof}
This result is a direct consequence of the last proposition. Indeed, by Proposition 7.15 in  \cite{JTquasi-Segal}, we know that if we have an adjunction between two model categories $F:M\rightleftharpoons N:G$ and we want to prove it is a Quillen adjunction, then we only need to prove that $F$ preserves cofibrations and $G$ preserves fibrant objects. 

 Now, the left Bousfield localization doesn't change cofibrations, and our adjunction was already a Quillen adjunction on the projective model structure; $j_!$ preserves cofibrations. We only have left to prove that $j^*$ preserves fibrant objects. But the fibrant objects in the dg-Segal model structure are the dg-Segal spaces that were fibrant in the projective model structure. By the last proposition, $j^*$ sends dg-Segal spaces to Segal spaces, and because it is already a Quillen adjunction in the projective model structure, it preserves fibrant objects in that structure. So it preserves fibrant objects in the dg-Segal model structure.

The adjunction  is a Quillen adjunction on the Segal and dg-Segal model structures, respectively, and we have finished our proof.
\end{proof}

\subsection{Complete dg-Segal spaces}

When trying to characterize $\infty$-categories using Segal spaces, we realize quickly that the definition of Segal spaces that has been already given is not enough. Indeed, there is a class of morphisms, called Dwyer-Kan morphisms, that should be equivalences but aren't. In \cite{ComSegalSpacesREZK}, in his quest to invert those, the author gets to the following result: 

\begin{defin}We define $E(1)$ to be the discrete space given at level $n\in\N$ by
$$E(1)_n=\{x,y\}^{[n]}, $$
i.e. by two non-degenerate cells on each level. Those are given by $(xy)^{n/2}$ and $(yx)^{n/2}$ if $n$ is odd and $(xy)^{(n-1)/2}$ and $(yx)^{(n-1)/2}$ if $n$ is odd.
\end{defin}

\begin{prop}\cite[Prop. 7.6]{ComSegalSpacesREZK}, \cite[Prop. 4.5]{Nima} Let $f:V\to W$ be a map between two Segal spaces. We assume that the morphism
$\Map(E(1),X)\to \Map(*,X)$
is a weak equivalence, with both $X=V$and $W$, for a certain morphism $ *\to E(1)$. Then $f$ is a Dwyer-Kan equivalence if and only if it is a weak equivalence.
\end{prop}

But that just means that the Dwyer-Kan equivalences are equivalences if they are between two $C$-local objects in $\sS$, where $C$ is a class with one object, $*\to E(1)$. So they define a new concept, that of complete Segal spaces, which are just the objects we have just defined.

\begin{defin}\cite[Def. 4.1, Section 6]{ComSegalSpacesREZK} \cite[Prop. 4.5]{Nima} Let $W$ be a Reedy fibrant simplicial space. We say that $W$ is a \textbf{complete Segal space} if the maps 
$$W_k\to \overbrace{W_1\times_{W_0}\cdots\times_{W_0}W_1}^{\text{$k$ times}} \text{  and  }\Map(E(1),W)\to\Map(*,W)$$
are weak equivalences for all $k\geq 1$. Or, equivalently, 
$$W_k\to \overbrace{W_1\times_{W_0}\cdots\times_{W_0}W_1}^{\text{$k$ times}} \text{  and  } W_0\to W_{hoequiv}$$
are weak equivalences for all $k\geq 1$ and $W_{hoequiv}$ the space of homotopy equivalences.
\end{defin}

And now Rezk uses the left Bousfield localization with $\{*\to E(1)\}$ as his $C$ and the complete Segal spaces as his $C$-local objects in order to have a model category that works for him.

\begin{teo}\cite{ComSegalSpacesREZK} There exists a simplicial closed model structure on the category of simplicial spaces, with the following properties:
\begin{enumerate}
	\item The cofibrations are precisely the monomorphisms.
	\item The fibrant objects are precisely the complete Segal spaces.
	\item The weak equivalences are precisely the Dwyer-Kan equivalences between complete Segal spaces.
\end{enumerate}
\end{teo}

In our case we will see that, mirroring the classic situation, the dg-Segal spaces we have defined are also not enough to completely characterize our dg-categories. Indeed, we can't prove that the functor $\Sing$ is fully faithful; we will be able to do so up to a certain morphism, that we will call a DK-equivalence. As such, we will need to do another left Bousfield localization in order to invert those, and then define its local objects to be our complete dg-Segal spaces.

Instead of doing it in that order, though, we will start with the definition of the complete dg-Segal model structure. For that, we will use the linearisation functor that we defined on the last section. 

\begin{defin}Let $E(1)$ be as defined above. We define
$$E_k=\L(j_!(E(1))). $$
\end{defin}

\begin{rem}We warn our readers of the fact that $E_k$ is not the image of a dg-category through our functor $\Sing$; in fact, it isn't even a dg-Segal space!
\end{rem}

\begin{teo}\label{Th: complee dg-Segal model structure}There exists a class of morphisms $C$ and a simplicial closed model structure on $\Fun^\S(\FreeS\op,\sS)$ such that
\begin{enumerate}
	\item The cofibrations are the same as in the projective model structure.
	\item The fibrant objects are the dg-Segal spaces which satisfy that $F(k)\to j^*(F)_{hoequiv}$ is a weak equivalence of simplicial spaces and that are fibrant for the projective model structure.
	\item The weak equivalences are the $C$-local equivalences.
\end{enumerate}
We call such a model structure \textbf{the complete dg-Segal model structure.}
\end{teo}

\begin{proof}
The idea of this proof is to transport the complete Segal model structure to our new setting via the linearisation functor. Indeed, we can construct this by using a left Bousfield localization on the dg-Segal model structure by the class $C=\{\Phi:E_k\to \Sing(k)\}=\{\Phi:j_!(E(1))\to j_!(*)\}$. We won't write down the details about whether this gives us an actual localization, as they are the exact same as in the construction of the dg-Segal model structure.

The only thing left to do is to see that the $C$-local objects are really as we have defined them above. Let $F$ be a $C$-local object. We have that $F$ is a dg-Segal space that is fibrant for the projective structure, and also that the following morphism
$$\Map(\Sing(k),F)\to\Map(E_k,F) $$
is a weak equivalence. By the definition of an adjunction, we have that 
$$\Map(E_k,F)=\Map(j_!(E(1)),F)\simeq \Map(E(1),j^*(F))\text{ and }\Map(\Sing(k),F)=\Map(j_!(*),F)\simeq \Map(*,j^*(F)).$$
So the condition is equivalent to asking for $\Map(*,j^*(F))\to \Map(E(1),j^*(F))$ to be a weak equivalence. And that is exactly the definition of $j^*(F)$ being a complete Segal space. So if $F$ is a $C$-local object, then $F$ is a dg-Segal space such that $j^*(F)_0=F(k)\to j^*(F)_{hoequiv}$ is a weak equivalence. By the construction of a Bousfield localization, a fibrant object for the complete dg-Segal model structure is such an $F$ which is also fibrant for the projective model structure.

We have all the conditions we needed and our proof is complete.
\end{proof}

\begin{rem}It is important to remember that $j^*(F)_{hoequiv}$ is the subset of $j^*(F)_1$ whose 0-simplexes are homotopy equivalences. Considering we've already seen that $j^*(F)_1=F(\D^1)$, we can rewrite $j^*(F)_{hocolim}$ as the subset of $F(\D^1)$ whose 0- simplexes are homotopy equivalences, with no mentions of the delinearisation morphism.
\end{rem}

\begin{nota}Let $F$ be a dg-Segal space. We denote the subset of  $F(\D^1)$ whose 0- simplexes are homotopy equivalences by $F_{hoequiv}$.
\end{nota}

\begin{defin}\label{Def: complete dg-Segal space}Let $F$ be an object in $\Fun^\S(\FreeS\op,\sS)$. We say that $F$ is \textbf{a complete dg-Segal space} if it is a dg-Segal space and the morphism
$$F(k)\to F_{hoequiv} $$
is a weak equivalence. 
\end{defin}

\begin{defin}We denote the full subcategory of $\dgS$ of complete dg-Segal spaces by $\dgSc$.
\end{defin}

As we did with the dg-Segal spaces, we will prove now that every element of $Sing$ is actually a \textit{complete} dg-Segal space.

\begin{prop}\label{complete dg-Segal in Sing}Let $T\in\dg$ be a dg-category. Then $\Sing(T)$ is a complete dg-Segal space.
\end{prop}

\begin{proof}
This is a direct consequence of the proof of Corollary 8.7 in \cite{Toen}. Indeed, during that proof Toën proves that for all $T\in\dg$, the morphism $\Map(k,T)=\Sing(T)(k)\to \Map(\D^1,T)=\Sing(T)(\D^1)$ induces an injection on $\pi_0$ and a bijection on $\pi_i$ for all $i>0$, and that its image in the homotopy category are the morphisms of $\l T\r$ that are isomorphisms. That means that the morphism is fully faithful, and that its essential image is $\Sing(T)_{hoequiv}$.

So for all $T$, the morphism $\Sing(T)(k)\to \Sing(T)_{hoequiv}$ is a weak equivalence, and $\Sing(T)$ is a complete dg-Segal space.
\end{proof}

Now, we have a nice definition of our $C$-local objects, but we still have the $C$-local equivalences defined in an abstract manner, and we said before that we were going to define something called the DK-equivalences that we needed to make into equivalences. Before we do that, though, we need one additional definition.

\begin{nota}We denote the full subcategory of cofibrant complexes of modules by $\Ch^c\subset \Ch$. We denote the full subcategory of perfect complexes by $\Ch^{perf}\subset \Ch$.
\end{nota}

\begin{rem}We remind our reader that the category of cochain complexes is accessible, and the perfect complexes are its compact objects. In consequence, for all $C\in\Ch$, we can write a filtered colimit $C=\colim E_i$, where $E_i\in\Ch^{perf}$ is a perfect complex for all $i$. 
\end{rem}

\begin{defin}Let $F\in\Fun^\S(\FreeS\op, \sS)$, let $x,y~\in\pi_0(F(k))$. Let $E\in\Ch^{perf}$ be a perfect complex. We define $F_{(x,y)}(E)$ to be the homotopy fiber of $F(E_{x,y})\to F(k)\times F(k)$ where $E_{x,y}$ is the free dg-category given by the graph with two objects, $x$ and $y$, and $E$ as the complex between $x$ and $y$,  
\begin{center}
	\begin{tikzcd}
		F_{(x,y)}(E)\ar[r]\ar[d]& *\ar[d]\\
		F(E_{x,y})\ar[r]& F(k)\times F(k).
	\end{tikzcd}
\end{center}
We define $F_{(x,y)}\in\Fun(\Ch^{c,op},\sS)$ to be, for all $C\in\Ch^c$ cofibrant complexes of modules, $$F_{(x,y)}(C)=\colim F_{(x,y)}(E_i),$$ where $C=\colim E_i$. 
We call $F_{(x,y)}$ \textbf{the dg-mapping space of $F$ at $x,y$}.
\end{defin}

These functors are always representable. In order to prove this, we will need a lemma.

\begin{lema}Let $F\in\dgS$ be a functor that satisfies the dg-Segal conditions, let $x,y\in\pi_0(F(k))$. Then the functor $F_{(x,y)}$ sends all homotopy colimits to homotopy limits.
\end{lema}

\begin{proof}
In order to make the writing of this proof easier, we will ignore the fact that $F$ is contravariant and call this "commuting with homotopy limits".

Now, proving that a functor commutes with homotopy limits can be done by proving that it commutes with all filtered homotopy limits, all homotopy pushouts, and all finite sums. As $F$ satisfies the dg-Segal conditions, it already commutes with all filtered limits along perfect objects, which means that $F_{(x,y)}$ commutes with all filtered limits; and by the third dg-Segal condition it already commutes with homotopy pushouts along generating cofibrations. As we can write every homotopy pushout as a filtered homotopy limit over a homotopy pushout along a generating cofibration (in this case, the projective model structure), that means that $F_{(x,y)}$ commutes with all homotopy pushouts. Finally, all finite sums are homotopy pushouts except for the null sum, which means that is all we have left to prove.

Let $G$ be a graph with two objects and the null complex between them. By definition, the free dg-category associated to $G$ is the coproduct $L(G)=k\coprod k$. By the first dg-Segal condition, that means that $F(L(G))=F(k\coprod k)\simeq F(k)\times F(k)$. If we apply now the definition of $F_{(x,y)}$, we have that $F_{(x,y)}(G)$ is the homotopy fiber of $F(L(G))=F(k)\times F(k)\to F(k)\times F(k)$. As this is an isomorphism, we have that $F_{(x,y)}(G)=*$, the null object in $\sS$, and we are done. 

We have proven, then, that $F_{(x,y)}$ sends all homotopy colimits to homotopy limits.
\end{proof}

\begin{prop}Let $F\in\dgS$ be a functor that satisfies the dg-Segal conditions, let $x,y\in\pi_0(F(k))$. There exists a unique complex of modules up to weak equivalence $F(x,y)\in\Ch$ such that $F_{(x,y)}(-)\simeq \Map(-,F(x,y))$.
\end{prop}

\begin{proof}
The uniqueness is just a consequence of the Yoneda lemma. We assume the existence of two complexes of modules, $F_1(x,y)$ and $F_2(x,y)$, such that $F_{(x,y)}(-)\simeq \Map(-,F_1(x,y))$ and also $F_{(x,y)}(-)\simeq \Map(-,F_2(x,y))$. But the Yoneda lemma tells us that a weak equivalence of representable presheaves must come from a weak equivalence on the representing objects. Which means that we have
$$F_1(x,y)\simeq F_2(x,y). $$
We have the uniqueness.

Let us see the existence next. This is a direct consequence of Proposition 1.9 in \cite{TV-Chern}, in particular of a certain point of the proof of part 1. Indeed, we prove there that if $A$ is an $\infty$-category, then all functor $G\in\Fun(A\op,\sS)$ that commutes with homotopy colimits is representable. That result has, then, reduced our problem to proving that the functor $F_{(x,y)}$ sends all homotopy colimits to homotopy limits (we remind the reader that $F_{(x,y)}$ is a contravariant functor). Now, by the previous lemma, we know that is the case. We have that $F_{(x,y)}$ is representable, i.e. there exists a complex of modules $F(x,y)$ such that $F_{(x,y)}(-)\simeq \Map(-, F(x,y))$.

And we have finished this proof.
\end{proof}

\begin{rem}By the Yoneda lemma again, it is obvious that if the functor $F$ is of the form $\Sing(T)$ with $T\in\dg$, then we have that for all $E\in\Ch^c$,
$$F_{(x,y)}(E)=\Map(E,T(x,y)). $$
\end{rem}

\begin{defin}Let $F\in\dgS$ be a dg-Segal space. We call the \textbf{homotopy category of $F$}, and we denote it by $\l F\r$, the category whose objects are the 0-simplexes of $F(k)$ and whose morphisms are given for all $x,y\in\pi_0(F(k))$ by 
$$\l F \r(x,y)=H^0(F(x,y)) $$
i.e. the cohomology groups of degree 0 of the complex of morphisms associated to the dg-mapping space at $x$ and $y$.
\end{defin}

\begin{rem}We point the reader to the fact that we have that if $F\in\dgS$, then
$$H^0(F(x,y))=\pi_0(\map_{j^*(F)}(x,y)), $$
where $\map_{j^*(F)}$ is the mapping space in the associated Segal space, so we can use the composition law there.
\end{rem}

We are now ready to define DK-equivalences.

\begin{defin}Let $f:F\to G\in\dgS$  be a morphism between two functors satisfying the dg-Segal conditions. We say that $f$ is \textbf{a DK-equivalence} if it satisfies the following conditions:
\begin{enumerate}
	\item The induced morphism $\l f\r:\l F\r\to \l G\r$ is essentially surjective.
	\item For all objects $x,y\in\pi_0(F(k))$, the induced morphism on the dg-mapping spaces, $F_{(x,y)}\to G_{(f(x),f(y))}$, is a quasi-equivalence of functors in $\Fun(\Ch^{c,op},\sS)$.
\end{enumerate}
\end{defin}

\begin{rem}\label{Rem: pseudo 2oo3}Some readers might ask themselves why we have defined DK-equivalences exclusively on dg-Segal spaces, and not on any functor. The reason is that if we do that, DK-equivalences would not satisfy the two-out-of-three condition. Indeed, if we have three morphisms $f$, $g$ and $f\circ g$, and two out of the three are DK-equivalences over any functor, then the third has the second condition of DK-equivalences because of the two-out-of-three condition on quasi-equivalences; but the first condition is not always true. It is only true if we have $f,g$ DK-equivalences or $g$ and $f\circ g$ DK-equivalences. It is true if the spaces are all dg-Segal, though.
\end{rem}

\begin{rem}If $F, G$ are of the form $F=\Sing(T), G=\Sing(T')$, what we just defined is exactly a weak equivalence between dg-categories.
\end{rem}

\begin{hyp}\label{Hyp: Re-zk}Let $f:F\to G$ be a morphism between two functors satisfying the dg-Segal conditions. Then $f$ is a DK-equivalence if and only if it is a weak equivalence in the complete dg-Segal model structure.  
\end{hyp}

We have now everything we need in order to prove that the functor is an equivalence.

\section{Hypercovers and free dg-categories}\label{hypercovers}

We will now take a slight detour from dg-Segal spaces to delve into hypercovers. Hypercovers are a useful concept, and it has been defined in several different contexts: simplicial sets and topological spaces, among others. Through the different definitions, some traits remain, though: in all of them, a hypercover of an object $A$ is an augmented object $U_*\to A$ such that a certain set of morphisms satisfies an epimorphism condition; and in most of them we have the property that $\hocolim U_n\simeq A$. Let us give the example of simplicial sets here, as it will be used substantially later.

\begin{defin}\label{Def: defin hypercovers sSet}\cite[Definition 6.5.3.2, Corollary 7.2.1.15]{HTT} Let $X\in\sS$ be a simplicial set, and $U_*\to X$ an augmented simplicial object in $\sS$. We say that $U_*$ is  \textbf{a hypercover of $X$} if for all $n\in\N$ the functor 
$$U_n\to U_*^{\partial\Delta^n}$$
is an effective epimorphism. In other words, $U_*\to X$ is a hypercover if for all $n\in\N$
$$\pi_0(U_n)\to \pi_0(U_*^{\partial\Delta^n})$$
is an epimorphism.
\end{defin}

\begin{prop}\label{Prop: hypercovers sSet} \cite[Theorem 6.5.3.12]{HTT} Let $X\in\sS$ be a simplicial set and $U_*\to X$ a hypercover of simplicial sets. Then we have that $\hocolim U_n\simeq X$.
\end{prop}

\begin{rem}Any reader familiar with Lurie's work will know that the definition there is given in the much greater generality of $\infty$-topoi. We have decided to translate it here to the language of simplicial sets because it is easier to grasp and it is the only context in which we will use it.
\end{rem}

Unluckily for us, none of the known definitions work for our situation, which means we will be forced to write our own. In this section, then, we will give a new definition of hypercovers that depend on a certain subcategory $M_0$. The next step will be proving that for all elements of a category we can construct a hypercover made entirely of objects in $M_0$. Of course, this cannot work for a general $M_0$: we will need a few extra conditions, which we will give before we give a construction of such a hypercover. 

As the reader can imagine, we decided to use the free dg-categories of finite type expecting them to be well-behaved enough that we can work with $\Free$-hypercovers; and after fixing the objects in order to keep it a reasonable size, we will prove that it can be done. Lastly, we will tackle the other condition we said nicely-behaved hypercovers have: the fact that the homotopy colimit gives us the original object. For that, we will first fix two objects and prove that we have such a condition on the morphism complexes, and then using the fact that we have fixed objects in the construction, expand it to hypercovers of dg-categories.

Let us start with the definition of $M_0$-hypercovers over a model category.

\begin{defin}Let $M$ be a model category and $M_0$ be a subcategory of $M$. Then we define $f:T\to T'$ a morphism in $M$ to be \textbf{an $M_0$-epimorphism} if for all $X\in M_0$ the induced functor 
$$\Map(X,T)\to \Map(X,T') $$
is an effective epimorphism in $\sS$, i.e. the morphism
$$\pi_0(\Map(X,T))\to \pi_0(\Map(X,T')) $$
is surjective.
\end{defin}

\begin{defin}Let $M$ be a model category and $M_0$ be a subcategory of $M$. Let $X\in M$ be an object of $M$, and $U_*\to X$ an augmented simplicial object in $M$. We say that $U_*$ is \textbf{an $M_0$-hypercover of $X$} if for all $n\in\N$ the functor 
$$U_n\to U_*^{\partial\Delta^n}$$
is an $M_0$-epimorphism. In other words, $U_*\to X$ is an $M_0$-hypercover if 
$$U_0\to X $$
is an $M_0$-epimorphism and for all $n\geq 1$
$$U_n\to (\R\cosk_{n-1}\sk_{n-1}U)_{n} $$
is an $M_0$-epimorphism.
\end{defin}

\begin{rem}It is easy to see that this definition recovers the classic definition of a hypercover in $\sS$: indeed, the definition we gave in Definition \ref{Def: defin hypercovers sSet} is just a $*$-hypercover, where $*$ is the initial object in $\sS$.
\end{rem}

 As usual in these cases, we will construct our hypercover level by level, so we need a definition of a $n$-truncated hypercover.

\begin{defin}Let $M$ be a model category and $M_0$ be a subcategory of $M$. Let $X\in M$ be an object in $M$. We define \textbf{an $n$-truncated $M_0$-hypercover of $X$} to be an augmented $n$-truncated simplicial set $X_*\to X$ where for all $i\leq n$
$$X_i\to X_*^{\partial\Delta^i}  $$
is an $M_0$-epimorphism.
\end{defin}

\begin{teo}\label{Th: existence hypercovers}Let $M$ be a model category where every object is fibrant, and let $M_0$ be a subcategory of $M$ which is closed for finite coproducts. We assume that for every $X\in M$ there exists an object in $M_0$, $U\in M_0$, and a morphism $U\to X$ that is an $M_0$-epimorphism. Then there exists an $M_0$-hypercover of $X$ consisting of objects in $M_0$.
\end{teo}

\begin{proof}
We are going to prove this by induction. Let $X\in M$ be an object in $M$. We will prove that for all $n\in \N$ there exists an $n$-truncated hypercover of $X$.

$\bullet\ \ n=0$. This is true by hypothesis. We have assumed that there exists an object $U_0\in M_0$ with a morphism $U_0\to X$ which is $M_0$-epi. We then take the $0$-truncated simplicial set which is $U_0$ on degree $0$. It is a $0$-truncated hypercover by definition.

 $\bullet\ \ n\in\N$. By induction hypothesis there exists an $n$-truncated hypercover of $X$ consisting of objects pf $M_0$, named $V_*\to X$. We have to construct an $(n+1)$-truncated hypercover of $X$, that we will call $U_*\to X$. 

We define $V_*'=\sk_{n+1}(\cosk_n V_*)$. This simplicial object is $(n+1)$-truncated, but the term $n+1$ is not necessarily in $M_0$. By hypothesis, there exists a morphism $U'_{n+1}\to V_{n+1}'$ that is a $M_0$-epi and such that $U'_ {n+1}\in M_0$. We define the $(n+1)$-truncated $M_0$-hypercover $U_*\to X$ to be $U_i=V_i'=V_i$ for all $i\leq n$ and 
$$U_{n+1}=U'_{n+1}\coprod_{i\leq n} V_i$$
for $n+1$, where we take the colimit over $i\leq n$. As the set of elements $i\leq n$ is finite and we have assumed that $M_0$ is closed for finite colimits, we still have that $U_{n+1}\in M_0$.

We need to prove that $U_*\to X$ is an augmented simplicial set, and also that we have the $M_0$-hypercover property. Let us construct the morphisms $U_m\to U_{n+1}$ and $U_{n+1}\to U_m$ for all morphism $[n+1]\to [m]$ and $[m]\to [n+1]$, with $m\leq n+1$.

These morphisms are straightforward: in one direction we take the composite
$$U_{n+1}\to U'_{n+1}\to V'_{n+1}\to V'_{m}=U_m. $$
and in the other sense the morphism given by the colimit.

And lastly, we need to prove that for all $i\leq n+1$ the morphism $U_i\to U_*^{\partial\Delta^i} $ is a $M_0$-epimorphism.

For $i\leq n$ this is true by induction hypothesis. Indeed, we have defined $U_*$ in such a way that $U_i=V'_i$, and we have that $V'_i=(\sk_{n+1}(\cosk_n V_*))_i=V_i$ for all $i\leq n$. We need then to prove that 
$$U_i=V'_i=V_i\to (\cosk_{i-1}\sk_{i-1}U_*)_i=(\cosk_{i-1}\sk_{i-1}V_*)_i $$
 is an $M_0$-epimorphism. Since $V_*$ is an $n$-truncated $M_0$-hypercover, this condition is verified. We only need to prove this for $n+1$.

 We need to prove, then, that the morphism
$$U_ {n+1}=U'_{n+1}\coprod V_i\to (U_*)^{\partial\Delta^n}=(\cosk_n\sk_n U_*)_{n+1}=(\cosk_n V_*)_{n+1}=V'_{n+1} $$
is an $M_0$-epimorphism. Let $W\in M_0$ be an object in $M_0$, we are going to prove that 
$$\Map(W,U'_{n+1}\coprod V_i)\to \Map(W,V'_{n+1})$$
is surjective on the $\pi_0$. Every object in $M$ is fibrant by hypothesis, which means that we can write $\Map(W,U'_{n+1}\coprod V_i)$ as $\Hom(C^*(W),U'_{n+1}\coprod V_i)$, and similarly for $\Map(W,V'_{n+1})$. We are, then, going to prove that
$$\pi_0(\Map(W,U'_{n+1}\coprod V_i))=\Hom(W,U'_{n+1}\coprod V_i)/\sim\to \pi_0(\Map(W,V'_{n+1}))=\Hom(W, V'_{n+1})/\sim $$
is surjective, where $\sim$ is the homotopy equivalence relation.

We will need an auxiliary morphism for this. By definition of a simplicial set, for all $i\leq n$ there exists a morphism $V_i=V'_i\to V'_{n+1}$. Also, by the way we have constructed $U_*$, there exists a morphism $u:U'_{n+1}\to V'_{n+1}$, which is an $M_0$-epimorphism. In consequence, by definition of a coproduct, there exists a morphism $f:U'_{n+1}\coprod V_i\to V'_{n+1}$, and a factorization $g:U'_{n+1}\to U'_{n+1}\coprod V_i$ such that $f\circ g=u$.

And now let us prove the surjectivity. Let $F:W\to V'_{n+1}$ be a morphism. As $u:U'_{n+1}\to V'_{n+1}$ is an $M_0$-epimorphism, that means that it exists, up to homotopy, a morphism $F':W\to U'_{n+1}$ such that $u\circ F'=F$. But we have said that $u$ factorizes through $f$. We can then, compute a morphism 
$$g\circ F':W\to U'_{n+1}\to U'_{n+1}\coprod V_i $$
such that $f\circ g\circ F'=u\circ F'=F$ up to homotopy. We have that the morphism
$$\pi_0(\Map(W,U'_{n+1}\coprod V_i))=\Hom(W,U'_{n+1}\coprod V_i)/\sim\to \pi_0(\Map(W,V'_{n+1}))=\Hom(W, V'_{n+1})/\sim $$
is an $M_0$-epimorphism. The $(n+1)$-truncated simplicial set $U_*\to X$ is an $(n+1)$-truncated $M_0$-hypercover of $X$ and we have finished the proof.
\end{proof}

Even though it would be possible to construct $\Free$-hypercovers directly, the number of objects in it would explode quite quickly, and we don't want that. So instead of working on the free dg-categories directly, we will first fix the objects.

\begin{prop}\label{Prop: free hyper} Let $X$ be a dg-category, and let $\O$ be its set of objects. Let $M$ be $\dg_\O$ the model category of dg-categories with $\O$ as a set of objects,  and $M_0=\Free_{\O}$ be the full subcategory of free dg-categories with $\O$ as a set of objects. There exists an $M_0$-hypercover $U_*\to X$ such that $U_i\in \Free_{\O}$ for all $i\in\N$.
\end{prop}

\begin{proof}
We just need to prove that the model category $M$ and the subcategory $M_0$ fulfill the conditions of Theorem \ref{Th: existence hypercovers}, i.e. that all object in $M$ is fibrant, that $M_0$ is closed for finite coproducts and that for all object $X\in M$ there exists an object $U\in M_0$ and a $M_0$-epimorphism $U\to X$. 

\begin{itemize}
	\item All dg-categories are fibrant, so this condition is fulfilled.
	\item Let $X,Y$ be two free dg-categories over $\O$. Then, by definition, there exists two graphs $X',Y'\in Gr(\Ch)$ such that $L(X')=X$ and $L(Y')=Y$. But $L$ is a left adjoint, and a finite coproduct is a special case of a colimit: as such, we know that $L$ commutes with finite coproducts. That gives us that $X\coprod Y=L(X')\coprod L(Y')=L(X'\coprod Y')$, and we have that $X\coprod Y$ is a free dg-category, coming from the coproduct of graphs $X'\coprod Y'$. We have that $\Free$ is closed for finite coproducts and we have finished.
	\item Let $X$ be a dg-category. We are going to prove that the morphism $LU(X)\to X$ is the $M_0$-epimorphism we need. First of all, by definition of a free dg-category, $LU(X)$ has the same objects as $X$ and it is an object in $\Free_{\O}$. Let $A$ be a free dg-category of finite type over $\O$. We need to prove, then, that 
	$$\Map(A,LU(X))\to \Map(A,X) $$
	is surjective on the $\pi_0$. But this morphism is already surjective. If we take $f:C^*(A)\to X$, we can define a morphism $f':C^*(A)\to LU(X)$ that gives us the right result. Indeed, $X$ and $LU(X)$ have the same objects, so that won't change. We only need to define the morphisms of complexes.
	
	Let $x,y\in C^*(A)$. Then we have a morphism $\phi:C^*(A)(x,y)\to X(f(x),f(y))$. Now, the morphisms in $L(U(X))$ are defined as follows:
$$L(U(X))(f(x),f(y))=\bigoplus_{m\in\N}\bigoplus_{x_1,\ldots, x_m\in\mathcal{O}}(U(X)(f(x),f(x_1))\otimes \cdots \otimes U(X)(f(x_m),f(y))). $$	
	 We define the morphism $\phi':C^*(A)(x,y)\to LU(X)(f(x),f(y)) $ such that for all $g\in C^*(A)(x,y)$, the image by $\phi'$ is the direct sum with $\phi(g)$ in the zero component and zero everywhere else.
\end{itemize} 
We have all three conditions for the existence of an $M_0$-hypercover. For all $X$ there exists a $\Free$-hypercover of $X$ composed of objects in $\Free$.
\end{proof}

There are a couple of things that we need to highlight from this construction. Let us give a definition first.

\begin{defin} We define a \textbf{split hypercover} of $E$ in $\Ch$ to be an augmented simplicial complex $E_*\to E$ such that for all $n\in\N$ the morphism $E_n\to E_*^{\partial\Delta^n}$ is a split epimorphism.
\end{defin}

\begin{rem} If we have a $\FreeS$-hypercover $U_* \to X$ defined like that then for all $x,y\in \Ob(X)$ and for all $n\in\N$ we have that $U_n(x,y)\to U_*(x,y)^{\partial\Delta^n}$ is a split epimorphism. But that is not all.
\end{rem}

\begin{coro}\label{Coro: split hyper} If we have an augmented simplicial object $T_*\to T$ in $\dg_O$ such that for all $x,y\in\O$ the associated augmented simplicial complex $T_*(x,y)\to T(x,y)$ is a split hypercover in $\Ch$, then $T_*\to T$ is a $\Free$-hypercover in $\dg_\O$.
\end{coro}

 When we described hypercovers, we gave one important property that those constructions tended to have: their homotopy colimits being quasi-equivalent to the original object. In order for that to happen in this context, we will need an auxiliary result beforehand.

\begin{lema}\label{Lem: hyper complexes} Let $M=\dg$ and $M_0=\Free$. Let $T$ be an object in $\dg$ and $T_*\to T$ a $\Free$-hypercover constructed using Theorem \ref{Th: existence hypercovers} and Proposition \ref{Prop: free hyper}. Then for all $x,y\in \Ob(T)$ we have that $\hocolim_{\Delta\op}(T_i(x,y))\simeq T(x,y)$ in $\Ch$.
\end{lema}

\begin{proof}

As we said in the last remark, if $T_*\to T$ is a $\Free$-hypercover constructed using the aforementioned theorem, then for all $x,y\in\O$, $T_*(x,y)\to T(x,y)$ is a split epimorphism. We will then prove that if $E_*\to E$ is a split hypercover of complexes, then $\hocolim_{\Delta\op}E_i \simeq E$.

We start with the connective case. Let us assume that $E_*, E\in\Ch^{\leq 0}$, and $E_*\to E$ is a split hypercover in $\Ch^{\leq 0}$. We can use the Dold-Kan equivalence, $DK:\sAb\rightleftharpoons \Ch^{\leq 0}:DK^{-1}$ combined with the forgetful functor $\mathcal{U}:\sAb\to \sS$. The morphism $\mathcal{U}DK^{-1}(E_*)\to \mathcal{U}DK^{-1}(E)$ is surjective over $\pi_0$ in $\sS$, i.e. it is a hypercover of simplicial sets. By Proposition \ref{Prop: hypercovers sSet} we have that $\hocolim_{\Delta\op}\mathcal{U}DK^{-1}(E_i)\simeq \mathcal{U}DK^{-1}(E)$ in $\sS$. The forgetful functor commutes with homotopy colimits over $\Delta\op$, so we have that $\hocolim_ {\Delta\op}DK^{-1}(E_i)\to DK^{-1}(E)$. But the Dold-Kan equivalence isn't just an equivalence: it is a model equivalence, meaning that it is an equivalence of categories that also preserves the model structure (\cite[Chapter III, Section 2]{JG}). In particular, if we apply $DK$ we have that $\hocolim_{\Delta\op}E_i\simeq E$ and we have finished.

We have then that the connective case is true. We will now reduce the general case to the connective case, using the naïve truncation $\tau:\Ch\to \Ch^{\leq 0}$ such that if $E\in\Ch$, $H^i(\tau(E))=H^i(E)$ for all $i\leq 0$ and  $H^i(\tau(E))=0$ for all $i>0$. But knowing that $E_*\to E$ is a split hypercover in $\Ch$ doesn't guarantee that $\tau(E_*)\to \tau(E)$ will also be a split hypercover in $\Ch^{\leq 0}$. To prove that, we need the following result.

\begin{slema}\label{SLem: h puissance} Let $E_*\to E$ be a split hypercover in $\Ch$ and $K$ a finite simplicial object. Then  we have the equivalence
$$H^i(E_*^K)\simeq H^i(E_*)^K. $$ 
\end{slema}
\begin{proof}
We are going to prove this by induction over the dimension of $K$. Let us prove that for all $n\in\N$, if $\dim K=n$, then $H^i(E_*^K)\simeq H^i(E_*)^K$.
\begin{itemize}
	\item $n=0$. This is almost immediate. Indeed, if $\dim K=0$, then $K=\coprod_p *$, and $E_*^K=E_0^p$. And as cohomology commutes with products, we have that $H^i(E_*^K)\simeq H^i(E_*)^K$. 
	\item $n\geq 1$. We assume that for every $K'$ with $\dim K'< n$ we have $H^i(E_*^{K'})\simeq H^i(E_*)^{K'}$. Let $K$ be a simplicial objet of dimension $n$. Then we have the following homotopy coproduct
\begin{center}
	\begin{tikzcd}
		\coprod\partial\Delta^n\ar[r]\ar[d]\arrow[dr,phantom, "\lrcorner^h", very near end]&\coprod \Delta^n\ar[d]\\
		K_{\leq n-1}\ar[r]&K
	\end{tikzcd}
\end{center}
where $K_{\leq n-1}$ is the simplicial subobject of $K$ of dimension $n-1$. In turn, that square gives us the following homotopy products 
\begin{center}
	\begin{tikzcd}
		E_*^K\ar[r]\ar[d]\arrow[dr,phantom, "\ulcorner^h", very near start]&E_*^{K_{\leq n-1}}\ar[d]\\
		\prod E_*^{\Delta^n}\ar[r]&\prod E_*^{\partial\Delta^n}
	\end{tikzcd}
	\begin{tikzcd}
		H^i(E_*)^K\ar[r]\ar[d]\arrow[dr,phantom, "\ulcorner^h", very near start]&H^i(E_*)^{K_{\leq n-1}}\ar[d]\\
		\prod H^i(E_*)^{\Delta^n}\ar[r]&\prod H^i(E_*)^{\partial\Delta^n}.
	\end{tikzcd}
\end{center}
Now, we would like to have that if we take the $H^i$ on the first square we still have a homotopy product. That is not true in general. But we know something extra about this square: indeed, as $E_*\to E$ is a split hypercover, we know that $\prod E_*^{\Delta^n}\to\prod E_*^{\partial\Delta^n}$ is a split epimorphism, and in that case we do have that the square 
\begin{center}
	\begin{tikzcd}
		H^i(E_*^K)\ar[r]\ar[d]\arrow[dr,phantom, "\ulcorner^h", very near start]&H^i(E_*^{K_{\leq n-1}})\ar[d]\\
		\prod H^i(E_*^{\Delta^n})\ar[r]&\prod H^i(E_*^{\partial\Delta^n})
	\end{tikzcd}
\end{center}
is a homotopy product. By induction, we have $H^i(E_*)^{K_{\leq n-1}}\simeq H^i(E_*^{K_{\leq n-1})}$ and $\prod H^i(E_*)^{\partial\Delta^n}\simeq \prod H^i(E_*^{\partial\Delta^n})$ (because $\dim K_{\leq n-1}=\dim \partial\Delta^n=n-1$); and by definition we have that $H^i(E_*)^{\Delta^n}\simeq H^i(E_n)\simeq H^i(E_*^{\Delta^n})$. We have then the following cubic diagram, where both the front and the back square are homotopy products and three out of four front-to-back arrows are equivalences.
\begin{center}
	\begin{tikzcd}[row sep=scriptsize,column sep=scriptsize]
		& H^i(E_*^K)\arrow[dl]\arrow[rr]\arrow[dd] & & H^i(E_*^{K_{\leq n-1}})\arrow[dl, "\sim"]\arrow[dd] \\
		H^i(E_*)^K\arrow[rr,crossing over]\arrow[dd] & & H^i(E_*)^{K_{\leq n-1}} \\
		& \prod H^i(E_*^{\Delta^n})\arrow[dl, "\sim"]\arrow[rr] &  & \prod H^i(E_*^{\partial\Delta^n})\arrow[dl, "\sim"] \\
		\prod H^i(E_*)^{\Delta^n}\arrow[rr] & & \prod H^i(E_*)^{\partial\Delta^n}\arrow[from=uu,crossing over]\\
	\end{tikzcd}
\end{center}
We have that the fourth arrow is also an equivalence, $H^i(E_*^K)\simeq H^i(E_*)^K$, and we have finished. 
\end{itemize}
\end{proof}
Now that we have that,  we can prove that if $E_*\to E$ is a split hypercover, then $\tau(E_*)\to \tau(E)$ is too. Indeed,  by Sublemma \ref{SLem: h puissance}, we have that for all $i\leq 0$,
$$H^i(\tau(E_*)^{\partial\Delta^n})\simeq H^i(\tau(E_*))^{\partial\Delta^n}\simeq H^i(E_*)^{\partial\Delta^n}\simeq H^i(\tau(E_*^{\partial\Delta^n})) $$
and $\tau(E_*)^{\partial\Delta^n}\simeq \tau(E_*^{\partial\Delta^n})$. We then have that for all $n\in\N,$ $\tau(E)_n\to \tau(E_*)^{\partial\Delta^n}\simeq \tau (E_*^{\partial\Delta^n})$ is a split epimorphism, and $\tau(E_*)\to \tau(E)$ is a split epimorphism.

Now, if $E_*\to E$ is a split hypercover, we have that for all $i\in\N$, $E_*[-i]\to E[-i]$ is also a split hypercover and $\tau(E_*[-i])\to \tau(E[-i])$ is too. By the connective case, we have that $\hocolim \tau(E_n[-i])\simeq \tau(E[-i])$. As this is true for all $i\in\N$, we have that $\hocolim E_n\simeq E$.

In conclusion, if $T_*\to T$ is a hypercover constructed using Theorem \ref{Th: existence hypercovers}, we have that $\hocolim(T_i(x,y))\simeq T(x,y)$ for all $x,y\in\Ob(T)$ in $\Ch$ and we have finished. 
 \end{proof}

We're almost ready to prove that if $T_*\to T$ is a $\Free$-hypercover constructed as instructed, then $\hocolim T_i\to T$ is a weak equivalence. We have proven that the complexes of morphisms have that condition. But does it transfer well from complexes to dg-categories and vice-versa?

\begin{lema}\label{Lem: forget commutes} The forgetful functor $U:\dg_\mathcal{O}\to Gr(\Ch)_\mathcal{O}$ commutes with homotopy colimits over $\Delta\op$.
\end{lema}

\begin{proof}
This is a direct consequence of \cite[Lemma 4.1.8.13.]{LurieHA}, in particular of its proof. Indeed, in that result we say that if we have a combinatorial monoidal model category $A$ and a small category $C$ such that its nerve $N(C)$ is sifted, and we have on one hand that $A$'s monoidal structure is symmetric and satisfies the monoid axiom, and on the other hand that $A$ is left proper and its cofibrations are generated by cofibrations between cofibrant objects; then, the forgetful functor $\Alg(A)\to A$ commutes with homotopy colimits over $C$.

Now, we take the monoidal structure on the graphs with fixed objects to be as follows: let $G,G'\in Gr(\Ch)_\O$ and $x,y\in\O$, we define
$$(G\otimes G')(x,y)=\oplus_z G(x,z)\otimes G'(z,y). $$
The category of dg-categories with fixed objects is the category of algebras over $Gr(\Ch)_\O$ with this monoidal structure, $\Alg (Gr(\Ch)_\O)$. As the category of graphs with this monoidal structure satisfies the above conditions and $N(\Delta)$ is sifted, we have that the forgetful functor $\dg_\O\to Gr(\Ch)_\O$ commutes with homotopy colimits and we have finished our proof.
\end{proof}

We are now ready to prove our result.

\begin{prop}\label{Prop: free colimitant} Let $T$ be a dg-category with fixed objects, and $T_*\to T$ a $\Free$-hypercover in $\dg_\O$, the category of dg-categories with fixed objects $\O=\Ob(T)$. Then we have that $\hocolim T_i\simeq T$ in $\dg$.
\end{prop}

\begin{proof}
Let $\phi:\hocolim T_*\to T$ be the morphism from the homotopy colimit to $T$. We need to prove that this morphism is a weak equivalence in $\dg_\O$. But as we have the same objects, the quasi-essential surjectivity is automatic. We only need to prove that $\phi$ is quasi-fully faithful. By definition, that means that for all $x,y\in\O$, we need to prove that $(\hocolim T_i)(x,y)\simeq T(x,y)$. But as we know from Lemma \ref{Lem: forget commutes} that the forgetful functor commutes with $\hocolim$, that is equivalent to asking that for all $x,y\in\O$, we have $\hocolim(T_i(x,y))\simeq T(x,y)$. By Lemma \ref{Lem: hyper complexes}, that is true. 

We have that $\phi:\hocolim T_i\simeq T$ in $\dg_\O$. If we can prove that the functor $\Phi:\dg_\O\to \dg$ commutes with homotopy colimits over $\Delta\op$, we have finished the proof. Let $X=\hocolim_{\Delta\op} T_i$ in $\dg_\O$, and let us call the forgetful morphisms $\Phi:\dg_\O\to \dg$ and $\Xi:\dg_\O\to \coprod_\O k/ \dg$, where for all $X\in \dg_\O$ its image $\Xi(X)$ gives us a morphism from $\coprod_\O k$ to $\Phi(X)$.

 For all $T'\in\dg$ we then have a morphism of the form
$$\Map(\Phi(X), T')\to \Map(\coprod k,T')\simeq \prod\Map(k,T'). $$
By getting the fiber of this morphism and doing the same thing with $\hocolim T_i$, we get the following diagram:
\begin{center}
	\begin{tikzcd}
		\Map(X,T')\ar[r]\ar[d,"\sim"]&\Map(\Xi(X),T')\ar[r]\ar[d, "f"]&\Map(\Phi(X),T')\ar[r]\ar[d]&\prod\Map(k,T')\ar[d,"="]\\
		\holim\Map(T_i,T')\ar[r]&\holim\Map(\Xi(T_i),T')\ar[r]&\holim\Map(\Phi(T_i),T')\ar[r]&\prod\Map(k,T').
	\end{tikzcd}
\end{center}

If we can prove that $f$ is a weak equivalence, we have finished. For that, we are going to use the far-left square of this diagram. Indeed, if the morphisms $\Map(\Xi,T')$ and $\holim \Map(\Xi,T')$ are weak equivalences, by 2-out-of-3 then $f$ will be a weak equivalence too. If we can prove that $\Xi$ is fully faithful, we will have everything we need. 

We have an adjunction $\Xi:\dg_\O\rightleftharpoons \coprod k/\dg:\Gamma$ where the right adjoint $\Gamma$ is such that for all $(K,\{x_\alpha\}_{\alpha\in\O}, x_\alpha\in\Ob(K))\in\coprod k/\dg$, the objects of $\Gamma(K)$ are $\O$ and for every $\alpha,\beta\in\O$, $\Gamma(K)(\alpha,\beta)=K(x_\alpha, x_\beta)$. It is easy to see that $\Gamma\Xi=\Id_{\dg_\O}$. We have then that $\Xi$ is fully faithful and then that $\Map(\Xi,T')$ is a weak equivalence.

As such, we have that 
\begin{center}
	\begin{tikzcd}
		\Map(X,T')\ar[r, "\sim"]\ar[d,"\sim"]&\Map(\Xi(X),T')\ar[d, "f"]\\
		\holim\Map(T_i,T')\ar[r,"\sim"]&\holim\Map(\Xi(T_i),T')
	\end{tikzcd}
\end{center}
and $f$ is a weak equivalence. So $\Phi$ commutes with colimits, and since $\hocolim T_i\simeq T$ in $\dg_\O$, we have $\hocolim T_i\simeq T$ in $\dg$ and we have finished this proof.
\end{proof}

We can now go back to our functor $\Sing$.

\section{Proving the equivalence}

We ended Section \ref{sec: section 1} by stating the following Hypothesis: 

\begin{hyp}\label{Hyp: Re-zk}Let $f:F\to G$ be a morphism between two functors satisfying the dg-Segal conditions. Then $f$ is a DK-equivalence if and only if it is a weak equivalence in the complete dg-Segal model structure.  
\end{hyp}

This is an important hypothesis. Indeed, even though everything we have stated up until now works without the need of the hypothesis, we cannot prove the equivalence of categories without it. Even with the hypothesis hanging, though, there is much we can do. In this section, we will first prove the fully faithfulness (up to a certain DK-equivalence being a weak equivalence) with the help of the hypercovers on dg-categories we defined in the last section. Then, we will prove the essential surjectivity by defining a concept of dg-Segal hypercover, and using it to define for all dg-Segal space $F$ a dg-category whose image by $\Sing$ is DK-equivalent to $F$.

\subsection{The functor $Sing$ is fully faithful}

Let us prove that $\Sing$ is, in fact, fully faithful. Most of the heavy lifting in this section will be done by the hypercovers in the last section. Indeed, we will prove that the restriction to $c\FreeS$ injects fully-faithfully, and then in a way "chop" every dg-category into free parts using the hypercover construction, as we can write every dg-category as a certain colimit of elements in $\Free$. We will finally prove that $\Sing$ commutes with those colimits up to DK-equivalence, and use the hypothesis to conclude. 

\begin{prop}\label{Prop: fully faithful for free}With the construction from Theorem \ref{Th: construction Sing_w}, the functor
$$\Sing^{c\Free}:\Ho(c\Free)\to \Ho(\Fun^\S(\FreeS\op,\sS)) $$
is fully faithful.
\end{prop}

\begin{proof}
For this, we use Theorems \ref{Th: equivalence for functor categories} and \ref{Th: fully-faith}, taking the model category to be $\FreeS$ and $W$ to be the weak equivalences in the full subcategory of $\dg$ of free categories of finite type. Indeed, if we go down to the homotopy category, we can factorize $\Sing^{c\Free}$ as
$$\Ho(c\Free)\to \Ho(\sS^{c\Free,W})\to \Ho(\Fun^\S(L_W\FreeS\op, \sS))=\Ho(\Fun^\S(\FreeS\op,\sS)). $$
But we know from Theorem \ref{Th: fully-faith} that the first morphism here is fully faithful, and from Theorem \ref{Th: equivalence for functor categories} that the second one is an equivalence. So the composition of the two is fully faithful, and we have finished our proof.
\end{proof}

We have finally everything needed in order to use our free hypercovers. Let us prove that we have the necessary DK-equivalences. For that, one important object will be the homotopy colimit of the image of a hypercover by $\Sing$. Although the image of a dg-category by the $\Sing$ functor is a complete dg-Segal space, we haven't proven that that property is closed under homotopy colimits. We will have, then, to start by proving that here it is in fact the case.

\begin{lema}\label{ch.2 representable colimit}Let $T_*\to T$ be a $\Free$-hypercover of $T$ constructed as in Theorem \ref{Th: existence hypercovers} and Proposition \ref{Prop: free hyper}. Let $x,y\in\pi_0(\Sing(T)(k))$ be two objects. Then the homotopy colimit of the dg-mapping space on $x$ and $y$ of its image by $\Sing$, $\hocolim(\Sing(T_i)_{(x,y)})$, is a representable functor.
\end{lema}

\begin{proof}
By the uniqueness of the Yoneda Lemma, asking for $\hocolim(\Sing(T_i)_{(x,y)})$ to be representable is equivalent to asking that there exists a complex of morphisms $T(x,y)\in\Ch$ such that 

$$\hocolim(\Map(-,T_i(x,y))\simeq \Map(-,T(x,y)). $$

But we know that $T_*\to T$ is a $\Free$-hypercover, which in particular means that for all $x,y\in T$, $T_n(x,y)\to T_*(x,y)^{\partial\Delta^n}$ is a split epimorphism. In consequence, we have that for all $E\in\Ch$,
$$\pi_0(\Map(E,T_n(x,y)))\to \pi_0(\Map(E,T_*(x,y)^{\partial\Delta^n}))\simeq \pi_0(\Map(E,T_*(x,y))^{\partial\Delta^n}) $$
is an epimorphism. But we have seen that this is exactly the definition of a hypercover in $\sS$. So we have that $\Map(E,T_*(x,y))$ is a hypercover of $\Map(E,T(x,y))$ in the simplicial sets, and by Proposition \ref{Prop: hypercovers sSet} that means that for all $E$ we have $\hocolim\Map(E,T_*(x,y))\simeq \Map(E,T(x,y))$. That is, by definition, the same as saying that 
$$\hocolim \Sing(T_i)_{(x,y)}=\hocolim(\Map(-,T_i(x,y))\to \Map(-,T(x,y)) $$
is a weak equivalence. The functor is representable and we have finished our proof.
\end{proof}

\begin{lema}\label{ch.2 hocolim dg-Segal}Let $T_*\to T$ be a $\Free$-hypercover of $T\in\dg$ constructed as in Theorem \ref{Th: existence hypercovers} and Proposition \ref{Prop: free hyper}. Then the homotopy colimit of its image by $\Sing$, $\hocolim(\Sing(T_i))$, satisfies the dg-Segal conditions.
\end{lema}

\begin{proof}
In order to prove this result, we need to prove that the homotopy colimit fulfills the three conditions of the definition of dg-Segal space. The first two are straightforward, and hinge on the fact that the homotopy colimit in this case commutes with finite products and the fact that all elements in the image of $\Sing$ satisfy the dg-Segal conditions.

1. Let $L,K$ be two free dg-categories of finite type. By the first dg-Segal condition on the image of $\Sing$, 
$$\hocolim(\Sing(T_i))(K\coprod L)=\hocolim(\Sing(T_i)(K\coprod L))\simeq \hocolim(\Sing(T_i)(K)\times\Sing(T_i)(L)).$$
And as the homotopy colimits over $\Delta\op$ commute with finite products, 
$$\hocolim(\Sing(T_i))(K\coprod L)\simeq \hocolim(\Sing(T_i)(K))\times\hocolim(\Sing(T_i)(L))$$
and we have finished.

2. The second property is trivial. Indeed,
$$\hocolim(\Sing(T_i))(\emptyset)\simeq\hocolim(\Sing(T_i)(\emptyset))\simeq\hocolim(*)\simeq *. $$
We arrive now to the third condition. For an issue of generality and also in order to lighten our notation, we will do the computations in a slightly larger context, with the help of this sublemma.

\begin{slema}\label{SLem: equivalent condition on dg-Segal}Let $F\in\Fun^\S(\FreeS^{c,op},\sS)$ be a functor that satisfies conditions 1 and 2 of the dg-Segal conditions, and also two additional conditions as follows:
\begin{enumerate}
 \setcounter{enumi}{3}
	\item For all $L, K\in\FreeS^c$, $F(L\coprod_{k\coprod k} K)\to F(L)\times_{F(k)\times F(k)} F(K)$ is a weak equivalence. 
	\item For all pairs of objects $x,y\in\pi_0(F(k))$, the associated dg-mapping space on $x$ and $y$ $F_{(x,y)}$ is representable.
\end{enumerate}
Then we have that $F$ satisfies condition 3 of the dg-Segal conditions. In other words, for all $G\in Gr(\Ch)$ a graph of finite type, $x,y\in \Ob(G)$ two objects in $G$, and $\alpha\in Z^n(G(x,y))$ a cycle in $G(x,y)$, the following diagram is a homotopy pullback
	\begin{center}
		\begin{tikzcd}
			F(L(G[\cancel{\alpha}]))\ar[r]\ar[d]\arrow[dr,phantom, "\ulcorner^h", very near start]& F(L(G))\ar[d]\\
			F(\D^c(1,n,1))\ar[r]& F(\D(1,n,1)).
		\end{tikzcd}
	\end{center}	
\end{slema}

\begin{proof}
In order to make this proof easier to follow, we start with the case where $G$ has two objects and move up from there. Assume $G$ is a free dg-category with two objects. There is a morphism $\gamma: F(G)\to \Map(\O_G,\O)$, where $\O_G=\Ob(G)$ is the set of objects of $G$ and $\O=\pi_0(F(k))$. We know that $\pi_0(\Map(\O_G,\O))$ is the set of morphisms from $\O_G$ to $\O$. Then, if we take $f:\O_G\to \O$, the fiber of $\gamma$ is $F_{(f(x),f(y))}(G)$.  By conflating the two points on $G$ with the two points on $\D^c(1,s,1)$ and $\D(1,s,1)$, we can reduce the problem to asking that for all $f:\O_G\to \O$ the following diagram is a homotopy pullback
	\begin{center}
		\begin{tikzcd}
			F_{(f(x), f(y))}(G[\cancel{\alpha}])\ar[r]\ar[d]& F_{(f(x),f(y))}(G)\ar[d]\\
			F_{(f(x),f(y))}(\D^c(1,n,1))\ar[r]& F_{(f(x),f(y))}(\D(1,n,1))
		\end{tikzcd}
	\end{center}
	 But by condition 5 the functor $F_{(x',y')}$ is representable for all $x',y'\in\O$, so $F_{(f(x),f(y))}(E_{x,y})\simeq \Map(E,F(f(x),f(y)))$, where $E_{x,y}$ is the free category given by two points and $E$ as the morphism of complexes from $x$ to $y$. So the question ends up being whether
	  	\begin{center}
		\begin{tikzcd}
			\Map(G[\cancel{\alpha}],F(f(x),f(y)))\ar[r]\ar[d]& \Map(G,F(f(x),f(y)))\ar[d]\\
			\Map(k^c[n],F(f(x),f(y)))\ar[r]& \Map(k[n],F(f(x),f(y)))
		\end{tikzcd}
	\end{center}
	is a homotopy pullback. Now, this diagram is in $\sS$, which is a proper model category, meaning that if one of the morphisms of this diagram is a fibration, we have finished. And indeed, the morphism $k[n]\to k^c[n]$ is a generating cofibration of $\Ch$, which means that the morphism $\Map(k^c[n],F(f(x),f(y)))\to \Map(k[n],F(f(x),f(y)))$ is a fibration. The diagram is a homotopy pullback and we have finished. If $G$ has two objects, $F$ satisfies the third dg-Segal condition on it.
	
	For the passage from two elements to more, it is a question of noticing that condition 3 is merely a local condition: intuitively, the only change being made in it is in relation to the morphisms between $x$ and $y$. Outside of that, it changes nothing whether the original free dg-category $G$ had two objects or three thousand. Following that logic, we will decompose $G$ in two sections: one that changes and one that does not.
	
	Let $G$ be a free dg-category of finite type, and let $x,y$ be two objects in $G$. It is easy to see that we can write it as $G=G^0\coprod_{k\coprod k}G_{x,y}$, where $G_{x,y}$ is a free dg-category with two objects and $G(x,y)$ as the morphism complex between those objects; and $G^0$ is a free dg-category such that $G^0$ has the same objects as $G$ and $G^0(x',y')=G(x',y')$ if $(x',y')\neq (x,y)$ and $G^0(x,y)=0$. As $G_{x,y}$ is a free dg-category with two objects, we have that $F$ satisfies the third dg-Segal condition on $G_{x,y}$. We can then construct a tower of homotopy pullbacks of the following form:
	\begin{center}
		\begin{tikzcd}
			F(G^0)\times_{F(k)\times F(k)}	F(G_{x,y}[\cancel{\alpha}])\ar[r]\ar[d]\arrow[dr,phantom, "\ulcorner^h", very near start]& F(G^0)\times_{F(k)\times F(k)}	F(G_{x,y})\ar[d]\\
			F(G_{x,y}[\cancel{\alpha}])\ar[r]\ar[d]\arrow[dr,phantom, "\ulcorner^h", very near start]& F(G_{x,y})\ar[d]\\
			F(\D^c(1,n,1))\ar[r]& F(\D(1,n,1))
		\end{tikzcd}
	\end{center}	
and the outside square is a homotopy pullback. By condition 4, we have that 
$$F(G^0)\times_{F(k)\times F(k)}	F(G_{x,y}[\cancel{\alpha}])\simeq F(G^0\coprod_{k\coprod k}G_{x,y}[\cancel{\alpha}])\simeq F(G[\cancel{\alpha}])$$ 
and 
$$F(G^0)\times_{F(k)\times F(k)}	F(G_{x,y})\simeq F(G^0\coprod_{k\coprod k}G_{x,y})\simeq F(G),$$ 
 so we have our condition 3 for all $G$ free dg-categories of finite type and we have finished our proof. 
\end{proof}
We now want to apply this sublemma to our homotopy colimit. We have already proven that $\hocolim(\Sing(T_i))$ satisfies conditions 1 and 2 of the dg-Segal conditions; we just need to prove that it also satisfies conditions 4 and 5 of the sublemma. Condition 5 is given by Lemma \ref{ch.2 representable colimit}, so let us check condition 4. Let $K, L$ be two free dg-categories of finite type. In a similar way than the proof of Proposition \ref{Prop: Sing is dg-Segal}, we have that 
$$\Sing(T_i)(K\coprod_{k\coprod k} L)\simeq \Sing(T_i)(K)\times_{\Sing(T_i)(k)^2}\Sing(T_i)(L).$$
But as long as the base of the finite product has a finite $\pi_0$, we have that the homotopy colimit commutes with it. Which means that
$$\hocolim(\Sing(T_i))(K\coprod_{k\coprod k} L)\simeq \hocolim(\Sing(T_i)(K)\times_{\Sing(T_i)(k)^2}\Sing(T_i)(L))\simeq$$
$$\simeq \hocolim(\Sing(T_i)(K))\times_{\hocolim(\Sing(k))^2}\hocolim(\Sing(T_i)(L)).$$
Hence the functor $\hocolim(\Sing(T_i))$ satisfies the condition 4 of the sublemma, and in consequence it satisfies the third condition of a dg-Segal space.

The functor $\hocolim(\Sing(T_i))$ satisfies all three dg-Segal conditions, and we have finished our proof.
\end{proof}

\begin{rem}We have proven here that if we have a functor with a representable dg-mapping space that satisfies the first two dg-Segal conditions and also that $F(L\coprod_{k\coprod k} K)\to F(L)\times_{F(k)\times F(k)} F(K)$, we have a dg-Segal space; and also we proved in Proposition \ref{Prop: dg-Segal are Segal} that if we have a dg-Segal space, then at the very least $F(\D^n)\simeq F(\D^1)\times_{F(k)}\cdots\times_{F(k)}F(\D^1)$. This resembles strongly the definition of a Segal space, but not precisely; indeed, we believe it is not possible to take out the representability condition on the last sublemma. As for the other implication, work done in this direction points to it not being possible to deduce condition 4 strictly from the dg-Segal conditions.
\end{rem}

Now that we know that $\hocolim(\Sing(T_i))$ is a dg-Segal space, we can see whether the induced morphism $\hocolim(\Sing(T_i))\to \Sing(T)$ is a DK-equivalence in $\dgS$.

\begin{lema}\label{Lem: hocolim DK}Let $T_*\to T$ be a $\Free$-hypercover of $T$ constructed as in Theorem \ref{Th: existence hypercovers} and Proposition \ref{Prop: free hyper}. Then the morphism $\hocolim(\Sing(T_i))\to \Sing(T)$ is a DK-equivalence.
\end{lema}

\begin{proof}
In order to prove this, we need to make sure this morphism satisfies both conditions of the definition. The first is easy: indeed, we have made sure by the construction of our hypercover that all terms of the hypercover have the same elements. So $\l \hocolim(\Sing(T_i))\r\to \l \Sing(T)\r$ is essentially surjective.

As for the condition on the morphisms, we have done all the work already. We know that for all $(x,y)\in\pi_0(\Sing(T)(k))$, $\Sing(T)_{(x,y)}(E)=\Map(E,T(x,y))$. So we can rewrite this condition as being whether
$$\hocolim(\Map(-,T_i(x,y))\simeq \Map(-,T(x,y)). $$

But this condition is whether $\hocolim(\Sing(T_i)_{(x,y)})$ is a representable functor. By Lemma \ref{ch.2 representable colimit}, that is true.

The second condition is fulfilled and the morphism $\hocolim(\Sing(T_i))\to \Sing(T)$ is a DK-equivalence. 
\end{proof}

But, the reader might be saying, we constructed our hypercovers over free dg-categories; and the results we have gotten about $\Sing$ being fully faithful only affect free dg-categories of finite type! And the reader would be right. But that is not a big issue: indeed, we can write every free dg-category as a filtered colimit of free categories of finite type.

\begin{lema}Let $G\in Gr(\Ch)$ be a graph on cochain complexes. Then there exists a filtered diagram $\phi: I\to Gr(\Ch)$ such that the image of $\phi$ is contained in $Gr(\Ch)^{tf}$ the graphs of finite type and such that $\colim_I \phi\simeq G$.
\end{lema} 

\begin{proof}
First of all, as perfect complexes are the compact objects in the category of cochain complexes, we have that for all $x,y\in \Ob(G)$, there exists a filtered colimit as follows
$$\Hom_G(x,y)\simeq \colim \Hom(x,y)_\alpha, $$
where $\Hom(x,y)_\alpha$ are all perfect complexes.\par
Now, let us construct the filtered category $I$ as follows.
\begin{itemize}
	\item The objects of $I$ are finite graphs $G_\beta$ of the following form:
	\begin{itemize}
		\item[$\bullet$] Its objects $\Ob(G_\beta)$ are finite subsets of $\Ob(G)$.
		\item[$\bullet$] Let $x,y\in\Ob(G_\beta)$ be two objects of $G_\beta$. Then we define the edge between  $x$ and $y$ to be $\Hom(x,y)_\alpha$ for some $\alpha$ in the filtered colimit.
	\end{itemize}	 
	\item Let $G_{\beta_1}$ and $G_{\beta_2}$ be two objects in $I$. Then a morphism $f:G_{\beta_1}\to G_{\beta_2}$ is defined as follows:
	\begin{itemize}
		\item[$\bullet$] It is the set inclusion on objects.
		\item[$\bullet$] Let $x,y\in\Ob(G_{\beta_1})$. Then the morphism on the edge $x,y$ is given by the map on the filtered category $\Hom(x,y)_{\alpha,\beta_1}\to\Hom(x,y)_{\alpha,\beta_2}$. 
	\end{itemize}
\end{itemize}
It is straightfoward to prove that this filtered category defines a filtered diagram $\phi:I\to Gr(\Ch)$ such that $G\simeq \colim_I G_\beta$, where all $G_\beta$ are graphs of finite type. We have finished our proof.
\end{proof}

Finally, we have enough information to prove the full faithfullness of our functor $\Sing$.

\begin{teo}\label{Th: fully faithfulness} Assuming Hypothesis \ref{Hyp: Re-zk} to be true, for all $T\in\dg$ we have $\Re(\Sing)(T)\simeq T$, and the functor $\Sing$ is fully faithful.
\end{teo}

\begin{proof}
Let $T$ be a dg-category. We construct, with the aforementioned methods, a $\Free$-hypercover of $T$, $T_i\to T$. By Lemma \ref{Lem: hocolim DK}, we know that the image by $\Sing$ of this morphism is a DK-equivalence. We then have the following diagram
	\begin{center}
		\begin{tikzcd}
			\Re(\hocolim(\Sing(T_i))\ar[rr,"f"]\ar[d,"g"]&& \Re(\Sing(T))\\
			\hocolim(\Re(\Sing(T_i)))\ar[d,"\simeq"]&& \\
			\hocolim(\Re(\Sing(\colim(L(G_\beta))))\ar[d,"\phi"]&&\\
			\hocolim(\colim(L(G_\beta)))\simeq\hocolim(T_i)\ar[uuurr, "\zeta"]& &
		\end{tikzcd}
	\end{center}	
Assuming Hypothesis \ref{Hyp: Re-zk} to be true, we know that the morphism $\hocolim(\Sing(T_i))\to \Sing(T)$ is a weak equivalence for the complete dg-Segal model structure. So $f$ is a weak equivalence on the homotopy categories. On the other hand, because $\Re$ is a left hand adjoint, we know it commutes with the homotopy colimit, so $g$ is also a weak equivalence. Lastly, we have proven in Proposition \ref{Prop: fully faithful for free} that $\Sing$ is fully faithful over free dg-categories of finite type, and as all $L(G_\beta)$ are free dg-categories of finite type by construction, we have that $\Re(\Sing(L(G_\beta)))\simeq L(G_\beta)$ for all $L(G_\beta)$ and the morphism $\phi$ is also a weak equivalence. By the two out of three condition on model categories, that means that the morphism $\zeta$ is also a weak equivalence.

That gives us the following diagram:
	\begin{center}
		\begin{tikzcd}
			\hocolim(T_i)\ar[r,"\zeta"]\ar[d,"~"]&\Re(\Sing(T_i))\ar[dl]\\
			T
		\end{tikzcd}
	\end{center}	
The vertical morphism is a weak equivalence by Proposition \ref{Prop: free colimitant}, and $\zeta$ is a weak equivalence too. By the two out of three condition, we have that $\Re(\Sing(T))\to T$ is a weak equivalence. The functor $\Sing$ is fully faithful and we have finished out proof.
\end{proof}

We finally know the functor is fully faithful! All we have left is to prove that it is essentially surjective.

\subsection{The functor $Sing$ is essentially surjective}

Let us prove that every functor that satisfies the dg-Segal conditions is isomorphic to an object of the form $\Sing(T)$. In order to do that, we are going to use hypercovers again: for all functor $F$ that satisfies the dg-Segal conditions, we will define a concept of hypercovers of a dg-Segal space. Even though the definitions are different, the constructions on hypercovers are very similar to the ones in Section 4. Once we have that, we will construct for all $F$ a hypercover made of functors of the form $\Sing(T_i)$ with $T_i$ a free dg-category. Lastly, we will construct the homotopy colimit $T$ of those free dg-categories, and prove that there always exists a DK-equivalence between the image by $\Sing$ of $T$ and $F$. Once that is done, the hypothesis gives us the essential surjectivity.

\begin{defin}Let $F, G\in\dgS$. We say that a morphism $f:F\to G$ is a \textbf{dg-Segal epimorphism} if we have the following two conditions:
\begin{enumerate}
	\item The morphism $f$ is an isomorphism on the objects, $\pi_0(F(k))\to \pi_0(G(k))$.
	\item For all $x,y,\in\pi_0(F(k))$, the induced morphism $F_{(x,y)}\to G_{(x,y)}$ is a split epimorphism.
\end{enumerate}
\end{defin}

Now that we have our definition of what epimorphism we want, we can define our hypercover. It isn't complicated: we will essentially use the same definition we used last time, utilizing dg-Segal epimorphisms instead of $M_0$-epimorphisms.

\begin{defin}Let $F\in\dgS$ a functor that satisfies the dg-Segal conditions, and $F_*\to F$ an augmented simplicial object in $\dgS$. We say that $F_*$ is \textbf{a dg-Segal-hypercover of $F$} if for all $n\in\N$ the functor 
$$F_n\to F_*^{\partial\Delta^n}$$
is a dg-Segal epimorphism. In other words, $F_*\to F$ is a dg-Segal-hypercover if 
$$F_0\to F $$
is a dg-Segal epimorphism and for all $n\geq 1$
$$F_n\to (\R\cosk_{n-1}\sk_{n-1}F)_{n} $$
is a dg-Segal epimorphism.
\end{defin}

As in section 4, we could prove that for every $F\in\dgS$, there exists a free dg-category $T$ such that $\Sing(T)\to F$ is a dg-Segal epimorphism. While that result is undoubtedly true, utilizing it would later make the objects in our hypercover explode, and we do not want that. So we are going to fix the objects first.

\begin{nota}We denote the functor $\Sing(k)\in\dgS$ by $\underline{k}$. 
\end{nota}

\begin{defin}Let $\O$ be a set. We define \textbf{the category of dg-Segal spaces with fixed objects over $\O$} to be the full subcategory of $\coprod_\O\underline{k}/\dgS$ of $F$ dg-Segal spaces such that the morphism $\O\to \pi_0(F(k))$ is an isomorphism, and we denote it by $\dgS_\O$.
\end{defin}

\begin{defin}Let $\Phi:\dgS_\O\to \dgS$ be the forgetful functor. We define a dg-Segal epimorphism on $\dgS_\O$ to be a morphism $f$ in $\dgS_\O$ such that $\Phi(f)$ is a dg-Segal epimorphism.
\end{defin}

\begin{rem}It is easy to see that the first condition of the dg-Segal epimorphism is always true in $\dgS_\O$. Consequently, we won't have to check that condition as long as we are working on fixed objects.
\end{rem}

\begin{lema}\label{Lem: existence dg-epi} Let $F\in\dgS$ be a functor that satisfies the Segal conditions. We fix a set $\O=\pi_0(F(k))$. There exists a free dg-category with fixed objects $T\in\dg_\O$ such that $\Sing(T)\to F$ is a dg-Segal epimorphism in $\dgS_\O$. 
\end{lema}

\begin{proof}

Most of the work for the construction of this dg-category has already been done, and we only have to put it together. We define a graph $G$ as follows:
\begin{itemize}
	\item $\Ob(G)=\O$.
	\item For all $x,y\in\O$, we have that $G(x,y)=F(x,y)$, the representing object of $F_{(x,y)}$.
\end{itemize}

We define then $T=L(G)$ as being the free category constructed from $G$.

Let us now prove that the morphism $\Sing(T)\to F$ is a dg-Segal epimorphism in $\dgS_\O$, or equivalently, that its projection on $\dgS$ is a dg-Segal epimorphism.

By the construction of $T=L(G)$, we have that for all $x,y\in\O$, 
$$L(G)(x,y)=\bigoplus_{m\in\N}\bigoplus_{(x_1,x_2, \ldots, x_m)\in\O^m}(G(x,x_1)\otimes\cdots \otimes G(x_m,y)), $$
and in particular, $G(x,y)=F(x,y)$ is a factor in $T(x,y)$. That means that for all $E\in\Ch$ there exists an inclusion $\Map(E,F(x,y))=F_{(x,y)}(E)\to\Map(E,T(x,y))$. It is easy to see, using the definition of homotopy fiber, that for all dg-category $T'$, $\Sing(T')_{(x,y)}=\Map(-,T'(x,y))$. Putting it all together, we have then that the morphism $\Sing(T)_{(x,y)}=\Map(-,T(x,y))\to F_{(x,y)}=\Map(-, F(x,y))$ is a split epimorphism.

The morphism $\Sing(T)\to F$ is a dg-Segal epimorphism and we have finished.
\end{proof}

Let us tackle now the hypercover result. 

\begin{teo}\label{Th: dg-Segal hypercover}Let $F\in\dgS$ be a functor that satisfies the dg-Segal conditions and let $\O=\pi_0(F(k))$. Then there exists a simplicial object $T_*$ in $\dg_\O$ such that $F_*=\Sing(T_*)$ and a morphism $F_*\to F$ such that $F_*\to F$ is a dg-Segal hypercover in $\dgS_\O$.  
\end{teo}

\begin{proof}
Again, this construction is almost identical to that of Theorem \ref{Th: existence hypercovers}, and as such we won't be going into much detail. We will construct our hypercover by induction, by proving that for all $n\in\N$ there exists an $n$-truncated dg-Segal hypercover with fixed objects of $F$ where every level is in the image of $\Sing$. 
\begin{itemize}
	\item $n=0$ is true by Lemma \ref{Lem: existence dg-epi}. There exists a dg-category with fixed objects $T\in\dg_\O$ such that $\Sing(T)\to F$ is a dg-Segal epimorphism, and that creates a 0-truncated dg-Segal hypercover.
	\item $n\in\N$. By induction hypothesis there exists an $n$-truncated dg-Segal hypercover of $F$, named $\Sing(T_*)\to F$. Let us construct an $(n+1)$-truncated hypercover of $F$, that we will call $\Sing(T_*)\to F$ too. As the construction doesn't change the first $n$ terms, there is no ambiguity in the notation. 

We define $V_*=\sk_{n+1}(\cosk_n \Sing(T_*))$. This simplicial set is $(n+1)$-truncated, but the term $n+1$ is not necessarily in the image of $\Sing$. By Lemma \ref{Lem: existence dg-epi} again, there exists a free dg-category $A$ with objects $\O$ and a morphism $\Sing(A)\to V_{n+1}$ that is a dg-Segal epimorphism. We define the $(n+1)$-truncated dg-Segal hypercover to be $\Sing(T_i)$ for all $i\leq n$ and 
$$\Sing(T_{n+1})=\Sing(A)\coprod \Sing(T_i)=\Sing(A\coprod T_i) $$
for $n+1$. This is possible because $\Sing$ commutes with finite coproducts.

With the same morphisms as in Theorem \ref{Th: existence hypercovers}, we have a simplicial object in $\dgS_\O$. We now just have to prove that for all $i\leq n+1$, $\Sing(T_i)\to \Sing(T_*)^{\partial\Delta^i}$ is a dg-Segal epimorphism, which by construction gets instantly reduced to proving that 
$$ \Sing(T_{n+1})=\Sing(A)\coprod \Sing(T_i)\to V_{n+1}$$
is a dg-Segal epimorphism in $\dgS_\O$, or equivalently, that for all $x,y\in\O$, $\Sing(T_{n+1})_{(x,y)}=\Map(-, T_{n+1}(x,y))\to V_{n+1 (x,y)}$ is a split epimorphism.

\end{itemize}
We have created $F_*\to F$ a dg-Segal hypercover with fixed objects such that for all $n\in\N$, there exists a free dg-category $T_n$ such that $\Sing(T_n)=F_n$. Now, we have proven in Theorem \ref{Th: fully faithfulness} that $\Sing$ is a fully faithful functor. In consequence, every morphism in the simplicial object $F_*$ comes from a morphism in $\dg_\O$, and there exists a simplicial object $T_*\in\dg_\O$ such that $\Sing(T_*)=F_*$. We have finished our proof.
\end{proof}

Now that we have constructed our hypercover $F_*\to F$, the last thing we need is to prove that for all $F\in\dgS$ there exists a dg-category $T$ such that $F$ is DK-equivalent to $\Sing(T)$, and we have our perfect candidate to do so. 

\begin{nota}Let $F\in\dgS$ be a functor satisfying the dg-Segal conditions. Let $\Sing(T_*)=F_*\to F$ be a dg-Segal hypercover with fixed objects as constructed in Theorem \ref{Th: dg-Segal hypercover}. We denote by $T$ the homotopy colimit of the simplicial object $T_*$. In other words, we define
$$T=\hocolim(T_i). $$
\end{nota}

We need now to prove that $\Sing(T)\to F$ is a DK-equivalence. In order to do that, we will use two DK-equivalences that are easier to prove: $\hocolim(\Sing(T_i))\to \Sing(\hocolim(T_i))=\Sing(T)$ and $\hocolim(\Sing(T_i))\to F$.

\begin{lema}\label{Lem: dg-Segal to complex hyper}Let $F_*\to F$ be a dg-Segal hypercover in $\dgS_\O$ with $\O=\pi_0(F(k))$ constructed as before. Then, for all $x,y\in\O$, the augmented simplicial complex $T_*(x,y)\to F(x,y)$ is a split hypercover in $\Ch$.
\end{lema}

\begin{proof}
This follows from the definition of dg-Segal hypercover.
\end{proof}

\begin{prop}\label{Prop: fixed objects dg-Segal}Let $F_*\to F$ be a dg-Segal hypercover constructed as in Theorem \ref{Th: dg-Segal hypercover}. Then the homotopy colimit of $F_*$ is DK-equivalent to $F$ in $\dgS_\O$.
\end{prop}

\begin{proof}
The proof of this follows very closely the one in Lemma \ref{Lem: hocolim DK}. Indeed, in order to prove something is a DK-equivalence, we have two conditions: firstly, that the morphism
$$\l\hocolim(F_i) \r\to \l F\r $$
is an essentially surjective. But the hypercover has been constructed to have fixed objects $\pi_0(F(k))$, so this condition is verified by construction.

That leaves us with the second condition. We fix $x,y\in \O$. Do we have that 
$$(\hocolim F_{i})_{(x,y)}\to F_{(x,y)} $$
is a quasi-equivalence? By definition, $F_{i,(x,y)}=\Map(-,T_i(x,y))$ and $F_{(x,y)}=\Map(-,F(x,y))$, so we can rewrite this condition as wondering whether 
$$\hocolim \Map(-,T_i(x,y))\to \Map(-,F(x,y))$$
is a quasi-equivalence. By Lemma \ref{Lem: dg-Segal to complex hyper}, we know that $T_*(x,y)\to F(x,y)$ is a split hypercover of complexes, so for all $n\in\N$, $T_n(x,y)\to T_*^{\partial\Delta^n}(x,y)$ is a split epimorphism, and in consequence for all $E\in\Ch$ the morphism
$$\pi_0(\Map(E,T_n(x,y)))\to \pi_0(\Map(E,T_*(x,y)^{\partial\Delta^n}))\simeq \pi_0(\Map(E,T_*(x,y))^{\partial\Delta^n})  $$
is an epimorphism. That means that $\Map(E,T_*(x,y))\to \Map(E,F(x,y))$ is a hypercover of simplicial sets, and by Proposition \ref{Prop: hypercovers sSet}, we have a quasi-equivalence
$$\hocolim\Map(-,T_i(x,y))\simeq \Map(-,F(x,y)). $$
So the morphism $\hocolim(F_i)\to F$ is a DK-equivalence and we have finished our proof.
\end{proof}

Like in the case of the dg-categories, we have proved that the homotopy colimit is DK-equivalent to $F$, but only with fixed objects. But if we want to use it, we need to see that it is indeed true in $\dgS$.

\begin{prop}Let $F$ be a dg-Segal space and $F_*\to F$ a dg-Segal hypercover constructed as in Theorem \ref{Th: dg-Segal hypercover}. Then the homotopy colimit of $F_*$ is DK-equivalent to $F$ in $\dgS$. 
\end{prop}

\begin{proof}
This proof is almost identical to the one in Proposition \ref{Prop: free colimitant}, and as such we won't spend too much time on its details. As was the case there, we have that the homotopy colimit of the hypercover is DK-equivalent to $F$, but only with fixed objects. So we're going to prove that the forgetful functor $\Phi:\dgS_\O\to \dgS$ commutes with those homotopy colimits. Let us denote the obvious forgetful functor by $\Xi: \dgS_\O\to \coprod \k/\dgS$. 

For all $G\in\dgS$ we then have a morphism of the form
$$\Map(\Phi(F), G)\to \Map(\coprod \k,G)\simeq \prod\Map(\k,G). $$
By getting the fiber of this morphism and doing the same with $\hocolim F_i$, we get the following diagram:
\begin{center}
	\begin{tikzcd}
		\Map(F,G)\ar[r]\ar[d,"\sim"]&\Map(\Xi(F),G)\ar[r]\ar[d, "f"]&\Map(\Phi(F),G)\ar[r]\ar[d]&\prod\Map(\k,G)\ar[d,"="]\\
		\holim\Map(F_i,G)\ar[r]&\holim\Map(\Xi(F_i),G)\ar[r]&\holim\Map(\Phi(F_i),G)\ar[r]&\prod\Map(\k,G)
	\end{tikzcd}
\end{center}

If we can prove that $f$ is a weak equivalence, we have finished. For that, let us prove that $\Xi$ is fully faithful. 

We have an adjunction $\Xi:\Ho(\dgS_\O)\rightleftharpoons \Ho(\coprod \k/\dgS):\Gamma$ where the right adjoint $\Gamma$ is such that for all $H\in\coprod \k/\dgS$ and $L\in\FreeS$, $\Gamma(H)(L)$ is the following homotopy pullback
\begin{center}
	\begin{tikzcd}
		\Gamma(F)(L)\ar[d]\ar[r]\arrow[dr,phantom, "\ulcorner^h", very near start]&F(L)\ar[d]\\
		\coprod_{\O_L} F(k)\ar[r]&\coprod_{\O_L}\O 
	\end{tikzcd}
\end{center}
where $\O_L=\Ob(L)$ is the set of objects of $L$. It is trivial that with such a construction, $\Gamma\Xi=\Id_{\dgS_\O}$. We have then that $\Xi$ is fully faithful and in consequence that $\Map(\Xi,G)$ is a weak equivalence. By the 2-out-of-3 condition, we have that $f$ is a weak equivalence, and so we have that $\Phi$ commutes with colimits. By Proposition \ref{Prop: fixed objects dg-Segal}, we know that $\hocolim(F_i)$ is DK-equivalent to $F$ in $\dgS_\O$, and that means that we have a DK-equivalence between $\hocolim(F_i)$ and $F$ in $\dgS$ too. We have finished the proof.
\end{proof}

Let us see now about $\Sing(T_i)\to \Sing(T)$. 

\begin{lema}Let $F_*\to F$ be a dg-Segal hypercover as constructed before, and let $T_*$ be a simplicial object in $\dg_\O$ with $\O=\pi_0(F(k))$ such that $\Sing(T_*)=F_*$. We define $T=\hocolim T_i$. Then the augmented simplicial object $T_*\to T$ is a $\Free$-hypercover of dg-categories.
\end{lema}

\begin{proof}
As we have fixed objects on this, by Corollary \ref{Coro: split hyper} if for all $x,y\in\O$, the augmented object $T_*(x,y)\to T(x,y)$ is a split hypercover of complexes, then $T_*\to T$ is a $\Free$-hypercover of dg-categories and we have finished.

Let us fix $x,y\in\O$. By Lemma \ref{Lem: dg-Segal to complex hyper}, $T_*(x,y)\to F(x,y)$ is a split hypercover. But as it is a split hypercover, by Lemma \ref{Lem: hyper complexes} we know that $T(x,y)=\hocolim(T_i(x,y))\simeq F(x,y)$. So the augmented object $T_*(x,y)\to T(x,y)$ is a split hypercover on $\Ch$, and in consequence $T_*\to T$ is a $\Free$-hypercover on $\dg$ and we have finished our proof.
\end{proof}

Now we just need to put everything together.

\begin{teo}\label{ch. 2: essential surjectivity}Let $F$ be a functor that satisfies the dg-Segal conditions. Then there exists a dg-category $T$ such that the morphism $\Sing(T)\to F$ is a DK-equivalence.
\end{teo}

\begin{proof}
Let $\Sing(T_*)=F_*\to F$ be a dg-Segal hypercover constructed following Theorem \ref{Th: dg-Segal hypercover}. We have then a diagram of this form:
	\begin{center}
		\begin{tikzcd}
			\Sing(\hocolim(T_i))=\Sing(T)\ar[rr]&&F\\
			\hocolim(\Sing(T_i))=\hocolim(F_i)\ar[u,"\phi"]\ar[rru,"\zeta"]
		\end{tikzcd}
	\end{center}	
We have proven that the corresponding augmented object $T_*\to T$ is a $\Free$-hypercover of dg-categories, and by Lemma \ref{Lem: hocolim DK}, the morphism $\phi$ is a DK-equivalence. We have also proven that $\hocolim(F_i)\to F$ is also a DK-equivalence, so by Remark \ref{Rem: pseudo 2oo3}, that means that our morphism $\Sing(T)\to F$ is also a DK-equivalence and we have finished.
\end{proof}

And we are done.

\begin{teo}\label{Th: equivalence}Assume Hypothesis \ref{Hyp: Re-zk} is true. Then, if we take the complete dg-Segal model structure on $\Fun^\S(\FreeS\op,\sS)$, the functor $\Sing: \dg\to \dgSc$ is essentially surjective. Hence there exists an equivalence of categories of the form
$$\Ho(\dg)\to \Ho(\dgSc). $$
\end{teo}

\section{Epilogue}

\subsection{Dealing with the hypothesis}\label{sec: hypothesis-solving}

We have thus seen that we can define  complete dg-Segal spaces, give their ambient functor category a model structure such that the complete dg-Segal spaces are its fibrant objects, and define a class of morphisms that would be the weak equivalences for such a model structure. We have also proven that assuming the DK-equivalences to actually be those weak equivalences, we have that complete dg-Segal spaces are a model for dg-categories. 

Now, the proof of that hypothesis is proving to be trickier than we expected. In \cite{ComSegalSpacesREZK}, Rezk gets around this issue by defining a completion functor that takes Segal spaces to complete Segal spaces, and such that the morphism from a Segal space to its completion is always a Dwyer-Kan equivalence. Then, the problem is reduced to proving that a Dwyer-Kan equivalence between two complete Segal spaces is a weak equivalence for the complete Segal model structure, and a relatively straightforward argument gives us the proof.

Unfortunately for us, the completion functor from Rezk's case doesn't work for us. Indeed, one integral part of its definition is the use of exponentials of simplicial sets, for which he uses the cartesian product. We cannot do that, as the lack of compatibility between monoidal and model structure in dg-categories means that the exponential is not well-defined in our case.

There are several possible ways to solve this issue. The best option, of course, would be to define a monoidal structure on the functor category that is compatible with the model structure. Indeed, doing this would not only solve our hypothesis, but also open the door to writing linear versions to classical category theory theorems that cannot be expressed yet. Unfortunately, the obvious answer, the convolution product, wouldn't work here. Indeed, although we can write our desired 
$$\Sing(T)\otimes\Sing(T')=\Sing(T\otimes T'), $$
it is easy to see that $T\otimes T'$ is not a free dg-category. This is, of course, not the only way to define a monoidal structure on a category, and other avenues are being explored. Another less ambitious option would be to abandon the idea of finding a monoidal model structure for the moment, and concentrate on defining a simplicial Hom.

\subsection{Refining the definition of complete dg-Segal spaces}\label{sec: Delta_k}

Although free dg-categories of finite type are enough for our definitions, they are probably not the minimal structure necessary in order to define complete dg-Segal spaces. Indeed, during our time working on this we have considered another, much smaller category, that we have called the linear simplicial category and denoted $\Delta_k$. This category $\D$ would be the full subcategory of $\dg$ consisting of dg-categories with a finite number of objects and such that for every two "consecutive" objects $i$ and $i+1$ the complex between them is concentrated on a single degree, where it is $k^n$ for a certain $n\in\N$. Such a category has several advantages over $\Free$: for one, all its objects are cofibrant for Tabuada's model structure. On top of that, we have that we can find an explicit description of the Hom from one object in $\D$ to a general dg-category, and also a description of weak equivalences in dg-categories via $\D$. Lastly, it seems to be a good candidate for our complete dg-Segal spaces: the image of the linearisation functor is entirely contained in $\D$.

Lastly, although this is still a work in progress, the author is convinced that managing to write complete dg-Segal spaces in terms of $\D$ would give us a relatively easy way to compare complete dg-Segal spaces to Mertens' own model for dg-categories in \cite{Arne-dg-quasi-cat}.

\subsection{Automorphisms of dg-categories}

As we said in the introduction, one of the potential uses of a construction such as the one of complete dg-Segal spaces would be to compute functors from dg-categories to other categories, or, in particular, the automorphisms of dg-categories. In his 2005 article, Toën proves that the group of automorphisms of $\infty$-categories is $\Z/2\Z$, consisting only of the identity and the dualisation functor. We expect the automorphism group of dg-categories to contain $\Z/2\Z$, of course, but not exclusively: at the very least, the group of automorphisms of the ring $k$ will have to be there, and also probably the Brauer group. 

It would also be interesting to take a look at higher homotopy groups of $Aut(\dg)$: we expect the Hochschild homology to appear at some point.

\backmatter

\bibliographystyle{smfalpha}
\bibliography{bibliographie} 

\end{document}